\documentclass[reqno]{amsart}

\usepackage{color}
\usepackage[pdfstartpage=1,colorlinks=true,bookmarks=false,pdfstartview={FitH}]{hyperref}
\usepackage{amsaddr}

\usepackage{amsmath,amssymb}
\usepackage{mathtools}
\usepackage{verbatim}
\usepackage{comment}
\newcommand{\COMMENT}[1]{}
\DeclarePairedDelimiter\abs{\lvert}{\rvert}
\makeatletter
\let\oldabs\abs
\def\abs{\@ifstar{\oldabs}{\oldabs*}}
\newcommand{\eps}{\varepsilon}

\newcommand{\R}{\mathbb{R}}

\newcommand{\Div}{\textnormal{div}\,}

\newcommand{\loc}{\textnormal{loc}}

\newcommand{\set}[1]{\left\{#1\right\}}
\newcommand{\absgrad}[1]{\abs{\nabla{#1}}}
\newcommand{\Bn}{\mathbb{B}^n}
\newcommand{\absgradB}[1]{\abs{\nabla_{\mathbb{B}^n}{#1}}}

\newcommand{\textif}{\text{ if }}
\newcommand{\textas}{\text{ as }}
\newcommand{\texton}{\text{ on }}

\newcommand{\textin}{\text{ in }}

\newcommand{\textand}{\text{ and }}

\newcommand{\C}{\mathcal{C}}

\newcommand{\dv}{dv_{g_{\Bn}}}
\newcommand{\LB}{\ensuremath{L^{\mathbb{B}^n}_{\gamma}}}

\newcommand{\Bp}{{\beta_+(\gamma)}}
\newcommand{\Bm}{{\beta_-(\gamma)}}
\newcommand{\Ap}{{\alpha_+(\gamma)}}
\newcommand{\Am}{{\alpha_-(\gamma)}}
\newcommand{\crits}{{\crits}}

\newtheorem{lem}{Lemma}[section]

\newtheorem{lemma}{Lemma}[section]

\newtheorem{theorem}{Theorem}
\newtheorem{prop}{Proposition}[section]

\newcommand{\bremark}{\begin{remark} \em}
\newcommand{\eremark}{\end{remark} }

\theoremstyle{definition}
\newtheorem{remark}{Remark}[section]

\def \crits {2^*(s)}

\numberwithin{equation}{section}

\definecolor{g2}{rgb}{0,0.6,0}
\definecolor{r2}{rgb}{0.8,0,0}

\vbadness=\maxdimen
\hbadness=\maxdimen

\begin{document}

\title{Mass and Extremals Associated with the Hardy--Schr\"odinger Operator on Hyperbolic Space }

\date{\today}


\author[H. Chan, N. Ghoussoub, S. Mazumdar, S. Shakerian]{{Hardy Chan, } {Nassif Ghoussoub,}
{Saikat Mazumdar  }       \&      {  Shaya Shakerian}}
\address{Department of Mathematics, The University of British Columbia \footnote{This work was done while Hardy Chan was completing his PhD thesis, and Saikat Mazumdar and Shaya Shakerian were
holding postdoctoral positions at the University of British Columbia, under the supervision of Nassif Ghoussoub. All were
partially supported by a grant from the Natural Sciences and Engineering Research Council of Canada.
During his visit to UBC, Luiz Fernando de Oliveira Faria was partially supported by CAPES/Brazil (Proc. 6129/2015-03).}}

\author[L.F.O. Faria]{Luiz Fernando de Oliveira Faria}
\address{Departamento de Matem\'{a}tica, Universidade Federal de Juiz de Fora}

\vspace{1mm}

\begin{abstract}
We consider the Hardy--Schr\"odinger operator $\displaystyle L_\gamma:=-\Delta_{\Bn}-\gamma{V_2}$ on the Poincar\'e ball model of the Hyperbolic space $\Bn$ ($n \geq 3$). Here $V_2$ is a well chosen radially symmetric potential, which behaves like the Hardy potential around its singularity at $0$, i.e., $V_2(r)\sim \frac{1}{r^2}$.
Just like in the Euclidean setting, the operator $L_\gamma$  is positive definite whenever $\gamma <\frac{(n-2)^2}{4}$, in which case we exhibit explicit solutions for the Sobolev critical  equation $L_\gamma u =V_{\crits}u^{\crits-1}$ in $\Bn,$
where $0\leq s <2$, $\crits=\frac{2(n-s)}{n-2}$, and  $V_{\crits}$ is a weight that behaves like $\frac{1}{r^s}$ around $0$. In dimensions $ n\geq5$, the above equation in a domain $\Omega$ of $\Bn$ containing $0$ and away from the boundary, has a ground state solution, whenever $0<\gamma\leq\frac{n(n-4)}{4}$, and provided $L_\gamma$ is replaced by a linear perturbation $L_\gamma-\lambda u$, where $\displaystyle \lambda  > \frac{n-2}{n-4} \left(\frac{n(n-4)}{4}-\gamma \right)$. On the other hand,
 in dimensions $3$ and $4$, the existence of solutions depends on whether the domain has a postive ``hyperbolic mass,''  
a notion that we introduce and analyze therein.
\end{abstract}

\maketitle

\section{Introduction} Hardy--Schr\"odinger operators on manifolds are of the form $\Delta_g -V$, where $\Delta_g$ is the Laplace--Beltrami operator and $V$ is a potential that has a quadratic singularity at some point of the manifold. For hyperbolic spaces, Carron \cite{Carron} showed that, just like in the Euclidean case and with the same best constant, the following inequality holds on any Cartan--Hadamard manifold $M$,
\begin{equation*}
\frac{(n-2)^{2}}{4} \displaystyle\int _{M}  \frac{u^{2}}{d_{g}(o,x)^{2}}\,dv_{g} \leq \displaystyle\int _{M} |\nabla_{g} u|^{2} \,dv_{g}  \qquad \hbox{ for  all } u \in C^{\infty}_{c}(M),
\end{equation*}
where $d_{g}(o,x)$ denotes the geodesic distance  to a fixed point $o \in M$.
There are many other works identifying suitable Hardy potentials, their relationship with the elliptic operator on hand, as well as corresponding energy inequalities \cite{AK, dAD, DFP, KO, LiWang, OT, YSK}. In the Euclidean case, the Hardy potential $V(x)=\frac{1}{|x|^{2}}$ is distinguished by the fact that  $\frac{u^{2}}{|x|^{2}}$ has the same homogeneity as $|\nabla u|^{2}$, but also  $\frac{u^{\crits}}{|x|^{s}}$, where $\crits=\frac{2(n-s)}{n-2}$ and $0\leq s<2$. In other words, the integrals $\displaystyle\int _{\R^{n}} \frac{u^{2}}{|x|^{2}}\,dx$, $\displaystyle\int _{\R^{n}} |\nabla u|^{2}\,dx$ and $\displaystyle\int _{\R^{n}} \frac{u^{\crits}}{|x|^{s}} \,dx$ are  invariant under the scaling $u(x) \mapsto \lambda^{\frac{n-2}{2}}u(\lambda x)  $, $\lambda >0$, which makes corresponding minimization problem non-compact, hence giving rise to interesting concentration phenomena. In \cite{Adi}, Adimurthi and Sekar use the fundamental solution of a general second order elliptic operator to generate natural candidates and derive Hardy-type inequalities. They also extended their arguments to Riemannian manifolds using the fundamental solution of the $p$-Laplacian. In \cite{DFP}, Devyver, Fraas and Pinchover study the case of a general linear second order differential operator  $P$ on  non-compact manifolds. They find a  relation between positive super-solutions of the equation $Pu=0$, Hardy-type inequalities involving $P$ and a weight $W$, as well as some properties of the spectrum of a corresponding weighted operator.

 In this paper, we shall  focus on the Poincar\'e ball model of the hyperbolic space $\Bn$, $n \geq 3$, that is the Euclidean unit ball  $B_{1}(0):= \{ x \in \R^{n}: |x|<1\}$  endowed with the metric
$g_{\Bn}=  \left( \frac{2}{1-|x|^{2}} \right)^{2}g_{_{\hbox{Eucl}}}.$
This framework has the added feature of radial symmetry, which plays an important role and contributes to the richness of the structure. In this direction, Sandeep and Tintarev   \cite{Sandeep--Tintarev} recently came up with  several integral inequalities involving weights on  $\Bn$ that are invariant under scaling, once restricted to the class of radial functions (see also Li and Wang \cite{LiWang}). As described below, this scaling is given in terms of the fundamental solution of the hyperbolic Laplacian $\Delta_{\Bn}u=\Div_{\Bn}(\nabla_{\Bn}u)$. Indeed, let
\begin{align}\label{hypgreen}
f(r):=\dfrac{(1-r^2)^{n-2}}{r^{n-1}} \quad \hbox{ and }\quad{G}(r):=\displaystyle\int_{r}^{1}f(t)\,dt,
\end{align}
where $r=\sqrt{\sum _{i=1}^{n} x_{i}^{2}}$  denotes  the Euclidean distance of a point $x \in B_{1}(0)$ to the origin.  It is known that $\frac{1}{n \omega_{n-1}}G(r)$ is a fundamental solution of the hyperbolic Laplacian $\Delta_{\Bn}$.
\medskip
\noindent
As usual, the Sobolev space $H^{1}(\Bn)$ is defined as the completion of $C^{\infty}_{c}(\Bn)$ with respect to the norm
$ \| u\|^{2}_{H^{1}(\Bn)} = \displaystyle\int _{\Bn} |\nabla_{\Bn}u|^{2}dv_{g_{\Bn}}.$
We denote by $H^{1}_{r}(\Bn) $ the subspace of radially symmetric functions. For functions $u \in H^{1}_{r}(\Bn)$, we consider the scaling
\begin{align}\label{hypscaling}
{u}_\lambda(r)=\lambda^{-\frac12}u\left(G^{-1}(\lambda{G}(r))\right), \qquad \lambda >0.
\end{align}
In \cite{Sandeep--Tintarev}, Sandeep--Tintarev have noted that for any $u\in{H}_r^1(\Bn)$ and $p\geq1$, one has the following invariance property:
\begin{equation*}
\displaystyle\int_{\Bn} |\nabla_{\Bn} u_\lambda|^{2} \,dv_{g_{\Bn}}=\displaystyle\int_{\Bn} |\nabla_{\Bn} u |^{2} \,dv_{g_{\Bn}} \,\,\textand\,\,\displaystyle\int_{\Bn}V_p \abs{u_\lambda}^p \,dv_{g_{\Bn}}=\displaystyle\int_{\Bn}V_p\abs{u}^{p}\,dv_{g_{\Bn}},
\end{equation*}
where
\begin{align}\label{Definition-V_p}
V_p(r):= \dfrac{f(r)^2(1-r^2)^2}{ 4(n-2)^2 G(r)^{\frac{p+2}{2}}}.
\end{align}
In other words, the hyperbolic scaling $r\mapsto{G}^{-1}(\lambda{G}(r))$   is quite analogous to the Euclidean scaling. Indeed, in that case, by taking $\overline{G}(\rho)=\rho^{2-n}$, we see that $\overline{G}^{-1}(\lambda\overline{G}(\rho))=\overline{\lambda}=\lambda^{\frac{1}{2-n}}$ for $\rho=\abs{x}$ in $\R^n$.
Also, note that $\overline{G}$ is --up to a constant-- the fundamental solution of the Euclidean Laplacian $\Delta$ in $\R^{n}$.
The weights $V_{p}$ have the following  asymptotic behaviors, for $n\geq3$ and $p>1$.
\[V_{p}(r)=\begin{cases}
\dfrac{c_0(n,p)}{r^{n(1-p/2^*)}}(1+o(1)) &\textas{r\to0},\\
\dfrac{c_1(n,p)}{(1-r)^{(n-1)(p-2)/2}}(1+o(1)) &\textas{r\to1}.
\end{cases}\]

In particular for $n \geq 3$, the weight $V_{2}(r)= \frac{1}{4(n-2)^{2}}\left( \frac{f(r)(1-r^{2})}{G(r)}\right)^{2} \sim_{ r \to0}\frac{1}{4r^{2}}$, and at $r=1$ has a finite positive value. In other words, the weight $V_{2}$ is qualitatively similar to the Euclidean Hardy weight, and Sandeep--Tintarev  have indeed established  the following Hardy inequality on  the hyperbolic space $\Bn$ (Theorem 3.4  of  \cite{Sandeep--Tintarev}). Also, see \cite{DFP} where they deal with similar Hardy weights.
\begin{equation*}
\dfrac{(n-2)^2}{4}\displaystyle\int_{\Bn}V_2\abs{u}^2\,dv_{g_{\Bn}}\leq\displaystyle\int_{\Bn}\absgradB{u}^2 \,dv_{g_{\Bn}} \quad \hbox{for any } u\in{H}^1(\Bn).
\end{equation*}
They also show in the same paper the following Sobolev inequality, i.e., for some constant $C>0$.
\begin{equation*}
\left(\displaystyle\int_{\Bn}V_{2^*} \abs{u}^{2^*}\, dv_{g_{\Bn}} \right)^{{2}/{2^*}}\leq{C}\displaystyle\int_{\Bn}\absgradB{u}^2\,dv_{g_{\Bn}} \quad \hbox{for any $u\in{H}^1(\Bn)$,}
\end{equation*}
\noindent where $2^*= \frac{2n}{(n-2)}$. By interpolating between these two inequalities taking $0\leq s\leq 2$, one easily obtain the following  Hardy--Sobolev inequality.   
\begin{lem}If $\gamma <\frac{(n-2)^2}{4}$, then there exists a constant $C>0$ such that, for any $u\in{H}^1(\Bn)$,
$$C\left(\displaystyle\int_{\Bn} V_{\crits} \abs{u}^{\crits}\,dv_{g_{\Bn}} \right)^{{2}/{\crits}}\leq \displaystyle\int_{\Bn}\absgradB{u}^2\,dv_{g_{\Bn}}-\gamma\displaystyle\int_{\Bn}V_2\abs{u}^2~ dv_{g_{\Bn}},$$
\noindent where $\crits:= \frac{2(n-s)}{(n-2)}$.
\end{lem}
 Note that, up to a positive constant, we have
$V_{\crits} \sim_{r \to 0 } \frac{1}{r^{s}}$, adding to the analogy with the Euclidean case, where we have for any $u\in{H}^1(\R^{n})$,
$$C\left(\displaystyle\int_{\R^n} \frac{\abs{u}^{\crits}} {|x|^s}\,dx \right)^{{2}/{\crits}}  \leq \displaystyle\int_{\R^n}
{\abs{\nabla u}}^2\,dx-\gamma \displaystyle\int_{\R^n} \frac{\abs{u}^{2}} {|x|^2}\,dx. $$
Motivated by the recent progress on the Euclidean Hardy--Schr\"odinger equation (See  for example Ghoussoub--Robert \cite{Ghoussoub--Robert-remaining cases, Ghoussoub--Robert CKN}, and the references therein), we shall consider the problem of existence of extremals for the corresponding best constant, that is
\begin{equation}\label{minpbm2}
\mu_{\gamma, \lambda }(\Omega):=\inf_{u\in{H_{0}^1(\Omega)}\backslash\set{0}}\dfrac{\displaystyle\displaystyle\int _{\Omega}\absgradB{u}^2\,dv_{g_{\Bn}}-\gamma\displaystyle\int _{\Omega} V_2 \abs{u}^2\,dv_{g_{\Bn}}- \lambda \displaystyle\int _{\Omega}  \abs{u}^2\,dv_{g_{\Bn}}}{\left(\displaystyle\displaystyle\int _{\Omega} V_{\crits} \abs{u}^{\crits} \,dv_{g_{\Bn}} \right)^{{2}/{\crits}}},
\end{equation}
where  $H_{0}^1(\Omega)$ is the completion of $C_{c}^{\infty}(\Omega)$ with respect to the norm $\| u\|= \sqrt{\displaystyle\int_{\Omega} |\nabla u|^{2}~ dv_{g_{\Bn}} }$.
Similarly to the Euclidean case, and once restricted to radial functions, the general Hardy--Sobolev inequality  for the hyperbolic  Hardy--Schr\"odinger operator is invariant under hyperbolic scaling described in \eqref{hypscaling},  
This invariance 
makes the corresponding variational problem non-compact and the problem of existence of minimizers quite interesting.
\medskip

In Proposition \ref{Prop:Explicit formula on Bn}, we start by showing that the extremals for the minimization problem \eqref{minpbm2}   in the class of radial functions $H^{1}_{r}(\Bn)$
can be written explicitly as:
\begin{equation*}
U(r)=c\left(G(r)^{-\frac{2-s}{n-2}\Am}+G(r)^{-\frac{2-s}{n-2}\Ap}\right)^{-\frac{n-2}{2-s}},
\end{equation*}
where $c$ is a positive constant and  $\alpha_\pm(\gamma)$ satisfy
$$\alpha_\pm(\gamma)= \dfrac{1}{2}\pm\sqrt{\frac{1}{4}-\frac{\gamma}{(n-2)^{2}}}.$$
In other words, we show that
\begin{equation}\label{minpbmrad}
\mu^{\textrm{rad}}_{\gamma, 0}(\Bn):=\inf_{u\in{H_{r}^1(\Bn)}\backslash\set{0}} \dfrac{\displaystyle\displaystyle\int_{\Bn}\absgradB{u}^2\,dv_{g_{\Bn}}-\gamma\displaystyle\int_{\Bn} V_2 \abs{u}^2\,dv_{g_{\Bn}}}{\displaystyle\left(\displaystyle\int_{\Bn} V_{\crits} \abs{u}^{\crits} \,dv_{g_{\Bn}} \right)^{{2}/{\crits}}}
\end{equation}
is attained by $U$.

\medskip

Note that  the radial function $G^{\alpha}(r)$ is a solution of $-\Delta_{\Bn}u-\gamma{V_2}u=0$ on $\Bn \setminus \{ 0\}$ if and only if  $\alpha = \alpha _\pm(\gamma)$. These solutions have the following asymptotic behavior
\begin{equation*}
G(r)^{\alpha_\pm(\gamma)}\sim c(n, \gamma) r^{-  \beta_\pm(\gamma)} \textas{r\to0},
\end{equation*}
where  $$\beta_\pm (\gamma)=\dfrac{n-2}{2}\pm\sqrt{\frac{(n-2)^2}{4}-\gamma}.$$
These then yield positive solutions to the equation
$$-\Delta_{\Bn}u-\gamma{V_2}u=V_{\crits}u^{\crits-1}\quad\textin\Bn.$$
\noindent
We point out the paper \cite{MS} (also see \cite{BS, BGGV, GS}), where the authors considered the counterpart of the Brezis--Nirenberg problem on $\Bn$ ($n \geq 3$), and discuss issues of  existence and non-existence  for the equation
$$-\Delta_{\Bn}u-\lambda u=u^{2^*-1}\quad\textin\Bn,$$
in the absence of a Hardy potential.

Next, we consider the attainability of $\mu_{\gamma, \lambda }(\Omega)$ in subdomains of $\Bn$ without necessarily any symmetry. In other words, we will search for positive solutions for the equation
\begin{equation}\label{main.eq}
\left\{ \begin{array}{@{}l@{\;}ll}
-\Delta_{\Bn}u-\gamma{V_2}u -\lambda u&=V_{\crits}u^{\crits-1} &\textin\Omega \\
\hfill u &\geq 0  &\textin\Omega  \\
\hfill u &=0   & \texton  \partial \Omega,
\end{array} \right.
\end{equation}
where $\Omega$ is a compact smooth subdomain of  $\Bn$ such that $ 0 \in \Omega$, but $\overline{\Omega}$ does not touch the boundary of $\Bn$ and $\lambda \in \R$. Note that the metric is then smooth on such $\Omega$, and the only singularity we will be dealing with will be coming from the Hardy-type potential  
$V_{2}$ and the Hardy--Sobolev weight $V_{\crits}$, which behaves like $\frac{1}{r^{2}}$
 (resp., $\frac{1}{r^{s}}$) at the origin.
This is analogous to the Euclidean problem on bounded domains considered by Ghoussoub--Robert \cite{Ghoussoub--Robert-remaining cases, Ghoussoub--Robert CKN}.
We shall therefore rely heavily on their work, at least in dimensions $n\geq 5$. Actually, once we perform a conformal transformation, the equation above reduces to the study of the following type of problems on bounded domains in $\R^n$:

\begin{equation*}  
\left\{ \begin{array}{@{}l@{\;}ll}
-\Delta v  - \left( \frac{\gamma}{|x|^{2}} + h(x)\right)  v &=b(x) \frac{v^{\crits-1}}{|x|^{s}} &\textin\Omega \\
\hfill v& \geq 0 & \textin\Omega\\
\hfill v &=0   & \hbox{ on }  \partial \Omega,
\end{array} \right.
\end{equation*}
where $b $ is a positive function in $C^1(\overline{\Omega})$  with
\begin{align}\label{def:b}
b(0)=\frac{(n-2)^{\frac{n-s}{n-2}}}{2^{2-s}} ~ \hbox{ and }~ \nabla b(0)=0,
\end{align}
and
\begin{align*}
h_{\gamma,\lambda}(x)= \gamma a(x)+\frac{4\lambda - {n(n-2)}}{(1-|x|^{2})^{2}}    ,
\end{align*}
\begin{align}\label{def:a}
a(x)= a(r)=
\left \{ \begin{array} {lc}
            \frac{4}{r} \left( \frac{1}{1-r}\right)+ \frac{4}{(1-r)^{2}} \qquad    &\hbox{ when }~ n=3, \\
           $~$\\
              8 \log \frac{1}{r}-4 +g_{4}(r) \qquad  &\hbox{ when }~ n=4,\\
           $~$\\
              \frac{4(n-2)}{n-4}+ r g_{n}(r) \qquad   &\hbox{ when }~ n\geq5.
            \end{array} \right.
\end{align}
with $g_{n}(0)=0$, for all $n \geq 4$.
Ghoussoub--Robert \cite{Ghoussoub--Robert-remaining cases} have recently tackled such an equation, but in the case where $h(x)$ and $b(x)$ are constants. We shall explore here the extent of which their techniques could be extended to this setting. To start with,
the following regularity result will then follow immediately.
\begin{theorem}[Regularity] \label{regularity} Let $\Omega \Subset \Bn$, $n \geq 3$,  and $\gamma <  \frac{(n-2)^{2}}{4}$.  If  $u \not\equiv 0$  is a non-negative  weak solution of the equation \eqref{main.eq} in  the hyperbolic Sobolev space $ H^{1}(\Omega)$, then
$$\lim _{ |x| \to 0}~ \frac{u(x)} {G(|x|)^{\alpha_{-}}} = K >0. $$
\end{theorem}

We also need to define a notion of mass of a domain associated to the operator $-\Delta_{\Bn}-\gamma{V_2}-\lambda$. We therefore show the following.

 \begin{theorem}[{The hyperbolic Hardy-singular mass of $\Omega \Subset \Bn$}]\label{hypermassI}
Let $0\in \Omega \Subset \Bn$, $n \geq 3$,  and $\gamma < \frac{(n-2)^{2}}{4}$. Let $\lambda \in \R$ be  such that the operator $-\Delta_{\Bn}-\gamma{V_2}-\lambda $  is coercive. Then, there exists a solution $K_\Omega \in C^{\infty} \left( \overline{\Omega} \setminus \{ 0\}   \right)$ to the linear problem,
\begin{eqnarray}
\left\{ \begin{array}{@{}l@{\;}ll}
-\Delta_{\Bn}K_\Omega-\gamma{V_2}K_\Omega -\lambda K_\Omega&=0 &\textin\Omega \\
\hfill K_\Omega &\geq 0  &\textin\Omega  \\
\hfill K_\Omega &=0   & \hbox{ on }  \partial \Omega,
\end{array} \right.
\end{eqnarray}
such that
$K_\Omega(x) \simeq_{|x| \to 0} c \ G(|x|)^{\alpha_{+}}$ for some positive constant $c$.
Furthermore,
\begin{enumerate}
\item If $K'_\Omega \in C^{\infty} \left( \overline{\Omega} \setminus \{ 0\}   \right)$ is another solution of the above linear equation, then  there exists a $C>0$ such that $K'_\Omega=C K_\Omega$. %
\item If $\gamma > \max \left\{ \frac{n(n-4)}{4},0 \right\}$, then  there exists $m^H_{\gamma, \lambda}(\Omega) \in \R$  such that
\begin{equation}
 K_\Omega(x)=G(|x|)^{\alpha_{+}} + m^H_{\gamma, \lambda}(\Omega) G(|x|)^{\alpha_{-}}+ o \left(G(|x|)^{\alpha_{-}} \right) \qquad \hbox{ as } x \to 0.
 \end{equation}
The constant $m^H_{\gamma, \lambda}(\Omega)$ will be referred to as the {\bf hyperbolic mass} of the domain $\Omega$ associated with the operator $\displaystyle -\Delta_{\Bn}-\gamma{V_2}-\lambda$.  
\end{enumerate}
\end{theorem}
\medskip

And just like the Euclidean case, solutions exist in high dimensions, while the positivity of the ``hyperbolic mass" will be needed for low dimensions. More precisely,

\begin{theorem}\label{Thm:Main Result}
 Let $\Omega \Subset \Bn$ $(n \geq 3)$ be a smooth domain with $0 \in \Omega $,  $0 \leq  \gamma < \frac{ (n-2)^{2}}{4}$ and let $\lambda \in \R$ be  such that the operator $-\Delta_{\Bn}-\gamma{V_2}-\lambda $  is coercive.
Then, the best constant $\mu_{\gamma, \lambda}(\Omega)$
is attained  under the following conditions:
\begin{enumerate}
\item $n >4$, $\gamma \leq \displaystyle  \frac{n(n-4)}{4}$ and
 $\displaystyle \lambda  > \frac{n-2}{n-4} \left(\frac{n(n-4)}{4}-\gamma \right)$.
\item $\max \left\{ \frac{n(n-4)}{4},0 \right\} <\gamma < \displaystyle  \frac{(n-2)^{2}}{4}$ and
 $m^H_{\gamma, \lambda}(\Omega)>0$.
\end{enumerate}
\end{theorem}

\medskip
\noindent
As mentioned above, the above theorem will be proved by
using a conformal transformation that reduces the problem to the Euclidean case, already considered by Ghoussoub--Robert \cite{Ghoussoub--Robert-remaining cases}. Actually, this leads to the following variation of the problem they considered, where the perturbation can be singular but not as much as the Hardy potential.

\begin{theorem}{\label{Thm:minexist:Euclidean}}
Let $\Omega $ be a bounded smooth  domain in $\R^{n}$, $n \geq 3$, with $0 \in \Omega $ and $0 \leq  \gamma < \frac{ (n-2)^{2}}{4}$. Let $\displaystyle  h \in C^1(\overline{\Omega} \setminus \{0\})$   be such that
\begin{equation}\label{eq:singh}
h(x) =  -  \C_{1} |x|^{-\theta}\log|x|+\tilde{h}(x)  \hbox { where } \lim _{x \to 0} |x|^{\theta}\tilde{h}(x) =  \C_{2} \hbox{ for  some } 0\le \theta <2 \hbox{ and }  \C_{1}, \C_{2} \in \R,
\end{equation}
and the operator $ - \Delta - \left(  \frac{\gamma }{|x|^2}+h(x)\right)$ is coercive. Also, assume that  $b(x)$ is a non-negative function in $ C^{1}(\overline{\Omega})$ with $b(0) > 0$. Then the best constant
\begin{equation}\label{eq:mu in Rn}
\mu_{\gamma,h}(\Omega) : =\inf_{u\in{H_{0}^1}(\Omega)\setminus\set{0}} \dfrac{\displaystyle\displaystyle\int _{\Omega} \left(\absgrad{u}^2 -\left( \frac{\gamma }{|x|^{2}}+ h(x) \right) u^{2}\right) \,dx}{\displaystyle\left(\displaystyle\int _{\Omega}  b(x) \frac{|u|^{\crits}}{|x|^{s}} \,dx\right)^{{2}/{\crits}}}
\end{equation}
is attained if one of the following two conditions is satisfied:
\begin{enumerate}
\item
$\gamma \leq  \frac{(n-2)^2 }{4} - \frac{(2-\theta)^2}{4}$ and, either $\C_{1}>0$ or $\{\C_{1}=0$, $\C_{2}>0\}$;
\item
$\frac{(n-2)^2 }{4} - \frac{(2-\theta)^2}{4}< \gamma < \frac{(n-2)^{2}}{4}$ and $m_{\gamma, h}(\Omega)>0$, where $m_{\gamma, h}(\Omega)$ is the mass of the domain $\Omega$ associated to the operator $  -\Delta- \left(\frac{\gamma}{|x|^{2}}+h(x)\right)$. 
\end{enumerate}
\end{theorem}

\section{Hardy--Sobolev type inequalities in hyperbolic space}
The starting point of the study of existence of weak solutions of the above problems are  the following  inequalities which will guarantee that the above
functionals are well defined and bounded below on the right function spaces. The Sobolev inequality for hyperbolic space \cite{Sandeep--Tintarev} asserts that for $n \geq 3$, there
exists a constant $C> 0$ such that
\begin{equation*}
\left(\displaystyle\int_{\Bn} V_{2^*} \abs{u}^{2^*}\,\dv\right)^{{2}/{2^*}}\leq{C}\displaystyle\int_{\Bn}\absgradB{u}^2\, \dv \quad \text{ for all } u\in{H}^1(\Bn),
\end{equation*}
where $2^* = \frac{2n}{n-2}$ and $V_{2^*} $ is defined in (\ref{Definition-V_p}). The Hardy inequality on $\Bn$ \cite{Sandeep--Tintarev} states:
\begin{equation*}
\dfrac{(n-2)^2}{4}\displaystyle\int_{\Bn} V_2 \abs{u}^2\, \dv \leq\displaystyle\int_{\Bn}\absgradB{u}^2\,\dv \quad \text{ for all } u\in{H}^1(\Bn).
\end{equation*}
Moreover, just like the Euclidean case, $\frac{(n-2)^2}{4}$  is the best Hardy constant in the above inequality on $\Bn$, i.e.,
$$\gamma_H := \frac{(n-2)^2}{4} = \inf_{u \in H^1(\Bn) \setminus \{0\}} \dfrac{\displaystyle\displaystyle\int_{\Bn}\absgradB{u}^2\, \dv}{\displaystyle\displaystyle\int_{\Bn} V_2 \abs{u}^2\, \dv}.$$
By interpolating these inequalities via H\"older's inequality, one gets the following Hardy--Sobolev inequalities in hyperbolic space.
\begin{lem}
\label{Lemma:hyperbolic Hardy--Sobolev Inequalities}
Let $\crits=\frac{2(n-s)}{n-2}$ where $0\leq s \leq 2$. Then, there exist a positive constant $C$  such that
\begin{equation}\label{Hardy--Sobolev:interpolation}
C\left(\displaystyle\int_{\Bn} V_{\crits} \abs{u}^{\crits}\, \dv\right)^{{2}/{\crits}}\leq \displaystyle\int_{\Bn}\absgradB{u}^2\, \dv \quad \text{ for all } u\in{H}^1(\Bn).
\end{equation}
If $\gamma < \gamma_H:= \frac{(n-2)^2}{4},$ then there exists $C_{\gamma}>0$ such that
\begin{equation}\label{HardySobolev}
{C_{\gamma}}\left(\displaystyle\int_{\Bn} V_{\crits} \abs{u}^{\crits}\, \dv\right)^{{2}/{\crits}}\leq \displaystyle\int_{\Bn}\absgradB{u}^2\, \dv-\gamma\displaystyle\int_{\Bn}V_2 \abs{u}^2\, \dv \ \text{ for all } u\in{H}^1(\Bn).
\end{equation}
\end{lem}
\begin{proof}
Note that for $s = 0$ (resp., $s = 2$) the first inequality is just the Sobolev (resp., the  Hardy) inequality in  hyperbolic space. We therefore have to only consider the case where $0 < s < 2$ where $\crits > 2.$ Note that
$\crits=\left(\dfrac{s}{2}\right)2+\left(\dfrac{2-s}{2}\right)2^*,$
and so
\begin{equation*}\begin{split}
V_{\crits}
&=\dfrac{f(r)^2(1-r)^2}{4(n-2)^2G(r)}\left(\dfrac{1}{\sqrt{G(r)}}\right)^{\crits}\\
&=\left(\dfrac{f(r)^2(1-r)^2}{4(n-2)^2G(r)}\right)^{\frac{s}{2}+\frac{2-s}{2}}\left(\dfrac{1}{\sqrt{G(r)}}\right)^{(\frac{s}{2})2+(\frac{2-s}{2})2^*}\\
&=\left(\dfrac{f(r)^2(1-r)^2}{4(n-2)^2G(r)}\left(\dfrac{1}{\sqrt{G(r)}}\right)^{2}\right)^{\frac{s}{2}}\left(\dfrac{f(r)^2(1-r)^2}{4(n-2)^2G(r)}\left(\dfrac{1}{\sqrt{G(r)}}\right)^{2^*}\right)^{\frac{2-s}{2}}\\
&=V_2^{\frac{s}{2}}V_{2^*}^{\frac{2-s}{2}}.
\end{split}\end{equation*}
Applying H\"older's inequality with conjugate exponents $\frac{2}{s}$ and $\frac{2}{2-s}$, we obtain
\begin{equation*}\begin{split}
\displaystyle\int_{\Bn} V_{\crits} \abs{u}^{\crits}\, \dv
&=\displaystyle\int_{\Bn}\left(\abs{u}^2\right)^{\frac{s}{2}}V_2^{\frac{s}{2}}\cdot\left(\abs{u}^{2^*}\right)^{\frac{2-s}{2}}V_{2^*}^{\frac{2-s}{2}}\,\dv\\
&\leq\left(\displaystyle\int_{\Bn} V_2 \abs{u}^2\,\dv\right)^{\frac{s}{2}}\left(\displaystyle\int_{\Bn} V_{2^*} \abs{u}^{2^*}\,\dv\right)^{\frac{2-s}{2}}\\
&\leq{C^{-1}}\left(\displaystyle\int_{\Bn}\absgradB{u}^2\,\dv\right)^{\frac{s}{2}}\left(\displaystyle\int_{\Bn}\absgradB{u}^2\, \dv \right)^{\frac{2^*}{2}\frac{2-s}{2}}\\
&=C^{-1}\left(\displaystyle\int_{\Bn}\absgradB{u}^2\,\dv\right)^{\frac{\crits}{2}}.
\end{split}\end{equation*}
It follows that for all $u \in H^1(\Bn),$
$$ \dfrac{\displaystyle\displaystyle\int_{\Bn}\absgradB{u}^2\, \dv  - \gamma \displaystyle\int_{\Bn} V_2 \abs{u}^2 \, \dv }{\displaystyle\left( \displaystyle\int_{\Bn} V_{\crits} \abs{u}^{\crits}\, \dv \right)^{{2}/{\crits}}} \ge \left(1 - \frac{\gamma}{\gamma_H} \right) \frac{\displaystyle\int_{\Bn}\absgradB{u}^2\, \dv }{\left( \displaystyle\int_{\Bn} V_{\crits} \abs{u}^{\crits}\, \dv\right)^{{2}/{\crits}}} .$$
Hence, (\ref{Hardy--Sobolev:interpolation}) implies  (\ref{HardySobolev}) whenever $\gamma < \gamma_H:= \frac{(n-2)^2}{4}.$
\end{proof}
The best constant $\mu_{\gamma}(\Bn)$ in inequality (\ref{HardySobolev}) can therefore  be written as:
\begin{equation*}
\mu_{\gamma} (\Bn)= \inf_{u \in H^1 (\Bn)\setminus \{0\}} \frac{\displaystyle\int_{\Bn}\absgradB{u}^2 \, \dv - \gamma \displaystyle\int_{\Bn} V_2 \abs{u}^2 \,\dv}{\left( \displaystyle\int_{\Bn} V_{\crits} \abs{u}^{\crits}\, \ dv_{\Bn} \right)^{{2}/{\crits}}}.
\end{equation*}
Thus,  any minimizer of $\mu_{\gamma}(\Bn)$
satisfies --up to a Lagrange multiplier-- the following Euler--Lagrange equation
\begin{equation}\label{eqgoal}
-\Delta_{\Bn}u-\gamma{V_2}u=V_{\crits}\abs{u}^{\crits-2}u,
\end{equation}
where $0\leq s<2$ and $\crits=\frac{2(n-s)}{n-2}$.

\section{The explicit solutions for Hardy--Sobolev equations on $\Bn$}

We first find the fundamental solutions associated to the Hardy--Schr\"odinger operator on $\Bn$, that is the solutions for the equation $-\Delta_{\Bn}u-\gamma{V_2}u=0$.
\begin{lem}
Assume $\gamma < \gamma_H:=\frac{(n-2)^2}{4}$. The fundamental solutions of
\begin{equation*}
-\Delta_{\Bn}u-\gamma{V_2}u=0
\end{equation*}
are given by
\begin{equation*}
u_{\pm}(r)=G(r)^{\alpha_\pm (\gamma)}\sim
\begin{cases}
\left(\dfrac{1}{n-2}r^{2-n}\right)^{\alpha_\pm(\gamma)}&\textas{r\to0},\\
\left(\dfrac{2^{n-2}}{n-1}(1-r)^{n-1}\right)^{\alpha_\pm (\gamma)}&\textas{r\to1},
\end{cases}
\end{equation*}
where
\begin{equation}\label{Relation between alpha and beta}
\alpha_\pm (\gamma) = \frac{\beta_\pm (\gamma)}{n-2}\quad \text{ and } \quad \beta_\pm (\gamma)=\dfrac{n-2}{2}\pm\sqrt{\frac{(n-2)^2}{4}-\gamma}.
\end{equation}
\end{lem}

\begin{proof}
We look for solutions of the form $u(r)=G(r)^{-\alpha}$. To this end we perform a change of variable $\sigma=G(r)$, $v(\sigma)=u(r)$ to arrive at the Euler-type equation
$$(n-2)^2v''(\sigma)+\gamma\sigma^{-2}v(\sigma)=0\quad\textin(0,\infty).$$
It is easy to see that the two solutions are given by $v(\sigma)=\sigma^{\pm}$, or $u(r)=c(n,\gamma)r^{-\beta_{\pm}}$ where $\alpha_{\pm}$ and $\beta_{\pm}$ are as in \eqref{Relation between alpha and beta}.
\end{proof}

\begin{remark}
We point out that $u_{\pm}(r)\sim{c}(n,\gamma)r^{-\beta_{\pm}(\gamma)}$ as $r\to0$.
\end{remark}

\begin{prop}\label{Prop:Explicit formula on Bn}
Let $-\infty<\gamma < \frac{ (n-2)^{2}}{4}$. The equation
\begin{equation}\label{Main problem: on Bn}
-\Delta_{\Bn}u-\gamma{V_2}u=V_{\crits}u^{\crits-1}\quad\textin\Bn,
\end{equation}
has a family of positive radial solutions which are given by
\begin{align*}
U(G(r))&=c\left(G(r)^{-\frac{2-s}{n-2}\Am}+G(r)^{-\frac{2-s}{n-2}\Ap}\right)^{-\frac{n-2}{2-s}}\\
 &= c\left(G(r)^{-\frac{2-s}{(n-2)^2}\Bm}+G(r)^{-\frac{2-s}{(n-2)^2}\Bp}\right)^{-\frac{n-2}{2-s}},
\end{align*}
where $c$ is a positive constant and  $\alpha_\pm(\gamma)$ and $\beta_\pm (\gamma)$ satisfy \eqref{Relation between alpha and beta}.

\end{prop}

\begin{proof}
With the same change of variable $\sigma=G(r)$ and $v(\sigma)=u(r)$ we have
$$(n-2)^2v''(\sigma)+\gamma\sigma^{-2}v(\sigma)+\sigma^{-\frac{\crits+2}{2}}v^{\crits-1}(\sigma)=0\quad\textin(0,\infty).$$
Now, set $\sigma=\tau^{2-n}$ and $w(\tau)=v(\sigma)$
$$\tau^{1-n}(\tau^{n-1}w'(\tau))'+\gamma\tau^{-2}w(\tau)+w(\tau)^{\crits-1}=0\quad\texton(0,\infty).$$
The latter has an explicit solution $$w(\tau)=c\left(\tau^{\frac{2-s}{n-2}\Bm}+\tau^{\frac{2-s}{n-2}\Bp}\right)^{-\frac{n-2}{2-s}},$$
where $c$ is a positive constant. This translates to the explicit formula
\begin{align*}
u(r)&=c \left(G(r)^{-\frac{2-s}{n-2}\Am}+G(r)^{-\frac{2-s}{n-2}\Ap}\right)^{-\frac{n-2}{2-s}}\\
& = c \left(G(r)^{-\frac{2-s}{(n-2)^2}\Bm}+G(r)^{-\frac{2-s}{(n-2)^2}\Bp}\right)^{-\frac{n-2}{2-s}}.
\end{align*}
\end{proof}

\begin{remark}
We remark that, in the special case $\gamma=0$ and $s=0$, Sandeep--Tintarev \cite{Sandeep--Tintarev} proved that the following minimization problem
$$\mu_{0}(\Bn)=\inf_{u\in{H}^1_r(\Bn)\setminus\{0\}}\dfrac{\displaystyle\displaystyle\int_{\Bn}\absgradB{u}^2\, \dv }{\displaystyle\displaystyle\int_{\Bn} V_{2^*} \abs{u}^{2^*}\, \dv}$$
is attained.
\end{remark}

\begin{remark}
The change of variable $\sigma=G(r)$ offers a nice way of viewing the radial aspect of hyperbolic space $\Bn$ in parallel to the one in $\R^n$ in the following sense.
\begin{itemize}
\item The scaling $r\mapsto{G}^{-1}(\lambda{G}(r))$ for $r=\abs{x}$ in $\Bn$ corresponds to $\sigma\mapsto\lambda\sigma$ in $(0,\infty)$, which in turn corresponds to $\rho\mapsto\overline{\lambda}\rho=\overline{G}^{-1}(\lambda\overline{G}(\rho))$ for $\rho=\abs{x}$ in $\R^n$, once we set $\overline{G}(\rho)=\rho^{2-n}$ and $\overline{\lambda}=\lambda^{\frac{1}{2-n}}$;
\item One has a similar correspondence with the scaling-invariant equations: if $u$ solves
    \begin{gather*}
    -\Delta_{\Bn}u-\gamma{V_2}u=V_{\crits}u^{\crits-1}\quad\textin\Bn,
    \end{gather*}
    then
    \begin{enumerate}
    \item as an ODE, and once we set  $v(\sigma)=u(r)$, $\sigma= G(r)$,  it is equivalent to
    \begin{equation}
    -(n-2)^2v''(\sigma)-\gamma\sigma^{-2}v(\sigma)=\sigma^{-\frac{\crits+2}{2}}v(\sigma)^{\crits-1}\quad\texton(0,\infty);
    \end{equation}
   \item as a PDE on $\R^n$, and by setting $v(\sigma)=u(\rho)$, $\sigma=\overline{G}(\rho)$, it is in turn equivalent to
    \begin{gather*}
    -\Delta{v}-\dfrac{\gamma}{\abs{x}^2}v=\dfrac{1}{\abs{x}^s}v^{\crits-1}\quad\textin\R^n.
    \end{gather*}
    \end{enumerate}
    This also confirm that the potentials $V_{\crits}$ are the ``correct'' ones associated to the power
    $|x|^{-s}$.
\item The explicit  solution $u$ on $\Bn$  is  related to the explicit solution $w$ on $\R^n$ in the following way:
$$u(r)=w\left(G(r)^{-\frac{1}{n-2}}\right).$$
\item Under the above setting, it is also easy to see the following integral identities:
    \[
    \begin{split}
    \displaystyle\int_{\Bn}\absgradB{u}^2\,dv_{g_\Bn}
    &=\displaystyle\int_{0}^{\infty}v'(\sigma)^2\,d\sigma\\
    \displaystyle\int_{\Bn}V_2{u}^2\,dv_{g_\Bn}
    &=\dfrac{1}{(n-2)^2}\displaystyle\int_{0}^{\infty}\dfrac{v^2(\sigma)}{\sigma^2}\,d\sigma\\
    \displaystyle\int_{\Bn}V_p{u}^p\,dv_{g_\Bn}
    &=\dfrac{1}{(n-2)^2}\displaystyle\int_{0}^{\infty}\dfrac{v^p(\sigma)}{\sigma^{\frac{p+2}{2}}}\,d\sigma,
    \end{split}
    \]
    which, in the same way as above, equal to the corresponding Euclidean integrals.
\end{itemize}
\end{remark}
\section{The corresponding perturbed Hardy--Schr\"odinger operator on Euclidean space}

We shall see in the next section that after a conformal transformation, 
the equation  \eqref{main.eq} 
is transformed into the Euclidean equation
\begin{equation} \label{HSeqn2}
\left\{ \begin{array}{@{}l@{\;}l@{\;}ll}
-\Delta u- \left( \frac{\gamma}{|x|^2}+h(x) \right) u&=&b(x)\frac{u^{\crits-1}}{|x|^s}  \ \ &\text{ in } \Omega,\\
\hfill u&>&0 &\text{ in } \Omega, \\
\hfill u&=&0 &\text{ on }\partial \Omega,
\end{array} \right.
\end{equation}
where $\Omega$ is a bounded domain in $\R^{n}$, $n \geq 3$, $h\in C^1(\overline{\Omega} \setminus \{0\})$ with $\lim\limits_{|x|\to0}|x|^{2}h(x)=0$ is such that the operator $-\Delta-\left(\frac{\gamma}{|x|^{2}}+h(x)\right)$  is coercive and $b(x)\in C^{1}(\overline{\Omega})$ is non-negative with $b(0) > 0$. The equation \eqref{HSeqn2} is the Euler--Lagrange equation for following energy functional on $D^{1,2}(\Omega)$,
$$J^{\Omega}_{\gamma, h}(u) :=\frac{\displaystyle\int _{\Omega} \left(\absgrad{u}^2 - \left( \frac{\gamma}{|x|^{2}}+ h(x) \right)  u^{2}\right) dx}{\left(~\displaystyle\int _{\Omega}  b(x) \frac{|u|^{\crits}}{|x|^{s}} \,dx\right)^{{2}/{\crits}}}.$$
Here $D^{1,2}(\Omega)$  -- or $H^1_0(\Omega)$ if the domain is bounded -- is the completion of $C^{\infty}_{c}(\Omega)$ with respect to the norm given by $||u||^{2}= \int \limits_{\Omega} |\nabla u|^{2}~dx$.
We let
$$\mu_{\gamma,h}(\Omega) : =\inf_{u\in{D^{1,2}}(\Omega)\setminus\set{0}}J^{\Omega}_{\gamma, h }(u)$$
A standard approach to find minimizers is to compare $\mu_{\gamma,  h}(\Omega)$ with $\mu_{\gamma,0}(\R^{n})$.  It is know that $\mu_{\gamma,0}(\R^{n})$ is attained when  $ \gamma \geq 0$,  are  explicit and take the form
\begin{equation*}\label{eq:Ue1}
U_{\eps}(x):=c_{\gamma,s}(n) \cdot \eps^{-\frac{n-2}{2}}U\left(\frac{x}{\eps}\right)=c_{\gamma,s}(n) \cdot \left(\frac{\eps^{\frac{2-s}{n-2}\cdot\frac{\Bp-\Bm}{2}}}{\eps^{\frac{2-s}{n-2}\cdot(\Bp-\Bm)}|x|^{\frac{(2-s)\Bm}{n-2}}+ |x|^{\frac{(2-s)\Bp}{n-2}}}\right)^{\frac{n-2}{2-s}}\quad 
\end{equation*}
for $x\in\R^n\setminus\{0\}$,
where $\eps>0$, $c_{\gamma,s}(n) > 0$,  and $\beta_{\pm}(\gamma)$ are defined in  \eqref{Relation between alpha and beta}. In particular, there exists $\chi>0$ such that
\begin{equation}\label{eq:U1}
-\Delta U_{\eps}-\frac{\gamma}{|x|^2}U_{\eps}=\chi \frac{U_{\eps}^{\crits-1}}{|x|^s}\hbox{ in }\R^{n}\setminus\{0\}.
\end{equation}
\bigskip

We shall start by analyzing the singular solutions and then
define the mass of a domain associated to the operator $ - \Delta - \left(\frac{\gamma }{|x|^2}+h(x)\right)$.

\begin{prop}
\label{Propoaition -Regularity}
Let $\Omega$ be a smooth bounded domain in $\R^n$  such that $0 \in \Omega$ and $\gamma <\frac{(n-2)^2 }{4}$. Let  $\displaystyle  h \in C^1(\overline{\Omega} \setminus \{0\})$ be such that $\lim\limits_{|x|\to0}|x|^{\tau}h(x)$ exists and is finite, for some $0\le \tau <2$, and that
the operator $ - \Delta - \frac{\gamma }{|x|^2}-h(x)$ is coercive. Then
\begin{enumerate}
\item There exists a solution $K \in C^\infty(\overline{\Omega} \setminus \{0\})$ for the linear problem
\begin{equation}\label{singsol}
\left\{\begin{array}{ll}
-\Delta K-\left(\frac{\gamma}{|x|^2}+h(x)\right)K=0 &\hbox{ in }\Omega\setminus\{0\}\\
\hfill K>0 &\hbox{ in }\Omega\setminus\{0\}\\
\hfill K=0 &\hbox{ on }\partial\Omega,
\end{array}\right.
\end{equation}
such that for some $c > 0$,
\begin{equation} K(x) \simeq_{x \to 0} \frac{c}{|x|^{\Bp}}.
\end{equation}
Moreover, if $K' \in C^\infty(\overline{\Omega} \setminus \{0\})$ is another solution for the above equation, then there exists $\lambda > 0$ such that $K' =\lambda K.$ \\
\item     Let   $\theta = \inf \{  \theta'\in[0,2): \lim\limits_{|x|\to0}|x|^{\theta'}h(x) \hbox{ exists and is finite}\}$. If $\gamma > \frac{(n-2)^2 }{4} - \frac{(2-\theta)^2}{4}$, then
there exists $c_1, c_2 \in \R  $ with $c_1 > 0$ such that
\begin{equation}\label{def:mass}
K(x) =\frac{c_1}{|x|^{\Bp}} +  \frac{c_2}{|x|^{\Bm}} +o \left(\frac{1}{|x|^{\Bm}}\right) \quad \text{ as } x \to 0.
\end{equation}
The ratio $\frac{c_2}{c_1}$ is independent of the choice of $K.$ We can therefore define the mass of  $\Omega$ with respect to the operator $ - \Delta - \left( \frac{\gamma}{|x|^2} + h(x) \right)$ as  $m_{\gamma,h}(\Omega):= \displaystyle \frac{c_2}{c_1}.$\\
\item
The mass $m_{\gamma,h}(\Omega)$ satisfies the following properties:
\begin{itemize}
\item
$m_{\gamma,0}(\Omega)<0$,
\item
If $h \leq h'$  and $h\not \equiv h'$,  then $m_{\gamma,h}(\Omega)<m_{\gamma,h'}(\Omega)$,
\item
If $\Omega' \subset \Omega$, then   $m_{\gamma,h}(\Omega')<m_{\gamma,h}(\Omega)$.
\end{itemize}
\end{enumerate}
\end{prop}

 \begin{proof}
The proof of (1) and (3) is similar to Proposition 2 and 4 in \cite{Ghoussoub--Robert-remaining cases} with only a minor change 
that accounts for the singularity of $h$. To illustrate the role of this extra singularity we prove (2). For that, we let $\eta \in C_c^\infty(\Omega)$ be such that $\eta(x) \equiv 1$ around $0.$ Our first objective is to write $K(x) := \frac{\eta(x)}{|x|^{\Bp}} + f(x)$ for some $f \in H_0^1(\Omega).$
Note that $  \gamma >  \frac{(n-2)^2}{4}- \frac{(2-\theta)^2}{4} \iff \beta_{+}-\beta_{-} < 2-\theta \iff 2 \beta_{+} < n-\theta$. Fix $\theta'$ such that $\theta< \theta'< \min \left\{\frac{2+\theta}{2}, 2-(\Bp-\Bm) \right\}$. Then $\lim\limits_{|x|\to0}|x|^{\theta'}h(x)$ exists and is finite.
\medskip

Consider the function
$$g(x) = - \left( - \Delta - \left( \frac{\gamma}{|x|^2} +h(x) \right) \right) (\eta |x|^{-\Bp}) \quad \text{ in } \Omega \setminus \{0\}. $$
Since $\eta(x) \equiv 1$ around $0$, we have that
\begin{equation}\label{Upper bound for g(x)-Mass}
|g(x)| \le  \left| \frac{h(x)}{|x|^{\Bp}} \right| \le C |x|^{-\left(\Bp+\theta' \right)} \quad \text{ as } x \to 0.
\end{equation}
Therefore $g \in L^{\frac{2n}{n+2}}(\Omega)$  if   $2\Bp +2 \theta'< n+2$, and  this holds since by our assumption $2 \beta_{+} < n-\theta$ and $2\theta' < 2+\theta$.  Since $L^{\frac{2n}{n+2}}(\Omega) = L^{\frac{2n}{n-2}}(\Omega)' \subset  H^1_0(\Omega)' $, there exists $f \in  H^1_0(\Omega)$ such that
$$- \Delta f - \left( \frac{\gamma}{|x|^2} +h(x) \right) f = g  \qquad \text{ in } H^1_0(\Omega).$$
By regularity theory, we have that $f \in C^2(\overline{\Omega} \setminus \{0\})$. We now show that
\begin{equation}\label{Existence of finite limit}
|x|^{\Bm} f(x)  \text{ has a finite limit as } x \to 0.
\end{equation}
Define
$ K(x) = \frac{\eta(x)}{|x|^{\Bp}} + f(x) \text{ for all } x \in \overline{\Omega} \setminus \{0\},$
and note that $K \in C^2(\overline{\Omega} \setminus \{0\})$ and is a solution to
$$- \Delta K - \left( \frac{\gamma}{|x|^2}+h(x) \right) K =0.$$
Write $g_+(x) := \max \{g(x), 0\}$ and $g_-(x) := \max \{-g(x), 0\}$ so that $g = g_+ - g_-,$ and let $f_1, f_2 \in H_0^1(\Omega)$ be weak solutions to
\begin{equation}\label{Solution to linear problem: f_1 - f_2}
- \Delta f_1 - \left( \frac{\gamma}{|x|^2} +h(x) \right) f_1 = g_+ ~\hbox{ and } - \Delta f_2 - \left( \frac{\gamma}{|x|^2} +h(x) \right) f_2 = g_-   \hbox{ in } H^1_0(\Omega).
\end{equation}
In particular, uniqueness, coercivity and the maximum principle yields $f = f_1 - f_2$ and $f_1,f_2 \ge 0.$ Assume that $f_1 \not\equiv 0$ so that $f_1 > 0 $ in $\Omega \setminus \{0\},$ fix $\alpha >   \Bp$ and $ \mu >0$. Define $u_-(x): = |x|^{-\Bm} + \mu |x|^{-\alpha}$ for all $x \neq 0$.  We then get that  there exists a small $\delta > 0$ such that
\begin{align}\label{Sub-Solution u_-}
\begin{split}
\left( - \Delta  - \left( \frac{\gamma}{|x|^2} +h(x) \right) \right)u_-(x)
&=   \mu \left(-\Delta -\frac{\gamma}{|x|^{2}} \right)|x|^{-\alpha} -\mu h(x)|x|^{-\alpha} -h(x)|x|^{-\Bm}\\
&=\frac{-\mu \left( \alpha -\Bp\right)\left(\alpha-\Bm \right)-|x|^{2}h(x)\left( |x|^{\alpha-\Bm} +\mu\right)}{|x|^{\alpha +2}}\\
& <0 ~\hbox{ for } x \in B_\delta(0) \setminus \{0\},
\end{split}
\end{align}
This implies that $u_-(x)$ is a sub-solution on $B_\delta(0) \setminus \{0\}$.  Let  $C > 0$  be such that $f_1 \ge C u_-$ on $\partial B_\delta(0)$. Since $f_1$ and $C u_- \in H^1_0(\Omega)$ are respectively super-solutions
and sub-solutions to $\left( - \Delta  - \left( \frac{\gamma}{|x|^2} +h(x) \right) \right)u(x) =0,$  it follows from the comparison principle (via coercivity) that  $f_1 > C u_- > C |x|^{-\Bm}$ on $ B_\delta(0) \setminus \{0\}$. It  then follows from  (\ref{Upper bound for g(x)-Mass}) that
\begin{align*}
g_+(x) \le  |g(x)| \le C |x|^{-\left(\Bp+\theta' \right)}
\le C_1 |x|^{ (2-\theta')-(\Bp-\Bm )}  \frac{f_1}{|x|^{2}}.
\end{align*}
Then rewriting
 \eqref{Solution to linear problem: f_1 - f_2}  as
$$- \Delta f_1 - \left( \frac{\gamma}{|x|^2} +h(x)+\frac{g_+}{f_{1}} \right) f_1 = 0
$$
yields
$$ - \Delta f_1  - \left( \frac{\gamma +  O\left(|x|^{ (2-\theta') - (\Bp-\Bm)} \right) }{|x|^2} \right) f_1  =0.$$
With our choice of $\theta'$ we can then conclude by the optimal regularity result  in  \cite[Theorem 8]{Ghoussoub--Robert-remaining cases}  that $|x|^{\Bm} f_1 $ has a finite limit as $x \to 0.$ Similarly one also  obtains that  $|x|^{\Bm} f_2$   has a finite limit as $x \to 0,$  and therefore (\ref{Existence of finite limit}) is verified.

It follows that there exists $c_2 \in \R $ such that
\begin{equation*}
K(x) =\frac{1}{|x|^{\Bp}} +  \frac{c_2}{|x|^{\Bm}} +o\left(\frac{1}{|x|^{\Bm}}\right) \quad \text{ as } x \to 0,
\end{equation*}
which proves the existence of a solution $K$ to the problem with the relevant asymptotic behavior. The uniqueness result 
yields the conclusion.
\end{proof}
\medskip

%

We now proceed with the proof of the  existence results, following again
 \cite{Ghoussoub--Robert-remaining cases}. We shall use the following standard sufficient condition for attainability.

\noindent

\begin{lemma}\label{low energy compactness} Under the assumptions of Theorem \ref{Thm:minexist:Euclidean}, if
\begin{align*}
\mu_{\gamma,h}(\Omega):= \inf_{u \in{H_{0}^1(\Omega})\setminus\set{0}} \frac{\displaystyle\int _{\Omega} \left(\absgrad{u}^2 -\left( \frac{\gamma}{|x|^{2}}+h(x) \right) u^{2}\right) dx}{\left(~\displaystyle\int _{\Omega}  b(x) \frac{|u|^{\crits}}{|x|^{s}} \,dx\right)^{{2}/{\crits}}} <\frac{ \mu_{\gamma,0}(\R^{n})}{b(0)^{2/\crits}},
\end{align*}
then the infimum $\mu_{\gamma,s}(\Omega)$ is achieved and equation \eqref{HSeqn2} has a solution.
\end{lemma}
\noindent {\bf Proof of Theorem \ref{Thm:minexist:Euclidean}:}
We will  construct a minimizing sequence $u_\eps$ in $ H^1_0(\Omega)\setminus\{0\}$ for the functional $J^{\Omega}_{\gamma,  h}$ in such a way that $\mu_{\gamma,  h}(\Omega)<b(0)^{-2/\crits}\mu_{\gamma,0}(\R^{n})$. As mentioned above, when  $ \gamma \geq 0$
the infimum $\mu_{\gamma,0}(\R^{n})$ is achieved, up to a constant,  by
the function
$$U(x):=\frac{1}{\left(|x|^{\frac{(2-s)\Bm}{n-2}}+|x|^{\frac{(2-s)\Bp}{n-2}}\right)^{\frac{n-2}{2-s}}}
\hbox{ for }x\in\R^{n}\setminus\{0\}.$$
In particular, there exists $\chi>0$ such that
\begin{equation}\label{eq:U2}
-\Delta U-\frac{\gamma}{|x|^2}U=\chi \frac{U^{\crits-1}}{|x|^s}\hbox{ in }\R^{n}\setminus\{0\}.
\end{equation}
Define a scaled version of $U$ by
\begin{equation}\label{eq:Ue2}
U_{\eps}(x):=\eps^{-\frac{n-2}{2}}U\left(\frac{x}{\eps}\right)=\left(\frac{\eps^{\frac{2-s}{n-2}\cdot\frac{\Bp-\Bm}{2}}}{\eps^{\frac{2-s}{n-2}\cdot(\Bp-\Bm)}|x|^{\frac{(2-s)\Bm}{n-2}}+ |x|^{\frac{(2-s)\Bp}{n-2}}}\right)^{\frac{n-2}{2-s}}\quad \hbox{for $x\in\R^{n} \setminus\{0\}$.}
\end{equation}
$\beta_{\pm}(\gamma)$ are defined in  \eqref{Relation between alpha and beta}. In the sequel, we write $\beta_+ := \Bp$ and $\beta_-:= \Bm.$
Consider a cut-off function $\eta\in C^\infty_c(\Omega)$ such that $\eta(x)\equiv1$ in a neighborhood of $0$ contained in $\Omega$.
\medskip

\noindent
{\bf {Case 1:}} Test-functions for the case when $\displaystyle \gamma\leq  \frac{(n-2)^2 }{4} - \frac{(2-\theta)^2}{4}$.
\medskip

\noindent
For $\eps>0$, we consider the test functions $u_{\eps}\in D^{1,2}(\Omega)$ defined by
$u_{\eps}(x):=\eta(x)U_{\eps}(x)$ for $x\in \overline{\Omega}\setminus\{0\}$.
To estimate $J_{\gamma,h}^{\Omega}(u_{\eps})$, we use the bounds on $U_{\eps}$ to obtain
\begin{align*}
\displaystyle\int _\Omega b(x) \frac{u_{\eps}^{\crits}}{|x|^s}\, dx &= \displaystyle\int _{B_\delta(0)} b(x)\frac{U_{\eps}^{\crits}}{|x|^s}~ dx+ \displaystyle\int _{ \Omega \setminus B_\delta(0)} b(x) \frac{u_{\eps}^{\crits}}{|x|^s}~ dx \\
&=\displaystyle\int_{B_{\eps^{-1}\delta}(0)}b(\eps x)\frac{U^{\crits}}{|x|^s}\, dx+   \displaystyle\int _{B_{\eps^{-1}\delta}(0)} b(\eps  x) \eta (\eps x )^{\crits} \frac{U^{\crits}}{|x|^s}\, dx\\
&=b(0) \displaystyle\int _{\R^{n}} \frac{U^{\crits}}{|x|^s}~ dx+ O \left(\eps^{ \frac{\crits}{2}  (\beta_{+}-\beta_{-})} \right).
\end{align*}
Similarly, one also has
\begin{align*}
\displaystyle\int _{\Omega}\left(|\nabla u_{\eps}|^2-\frac{\gamma}{|x|^2}u_{\eps}^2\right)\, dx &=  \displaystyle\int _{B_\delta(0)}\left(|\nabla U_{\eps}|^2-\frac{\gamma}{|x|^2}U_{\eps}^2\right)~ dx+   \displaystyle\int _{\Omega \setminus B_\delta(0)}\left(|\nabla u_{\eps}|^2-\frac{\gamma}{|x|^2} u_{\eps}^2\right)~ dx\\
&= \displaystyle\int _{B_{\eps^{-1}\delta}(0)} \left(|\nabla U|^2-\frac{\gamma}{|x|^2}U^2\right)~ dx+ O\left(\eps^{\beta_{+}-\beta_{-}}\right)\\
&= \displaystyle\int _{\R^{n}} \left(|\nabla U|^2-\frac{\gamma}{|x|^2}U^2\right)~ dx+ O\left(\eps^{\beta_{+}-\beta_{-}}\right)\\
&=\chi  \displaystyle\int _{\R^{n}} \frac{U^{\crits}}{|x|^s}~ dx+ O\left(\eps^{\beta_{+}-\beta_{-}}\right).
\end{align*}
\noindent
Estimating  the lower order terms as $\eps\to 0$ gives
\begin{equation*}
\displaystyle\int _{\Omega} \tilde{h}(x)u_{\eps}^2~dx  =
 \left \{ \begin{array} {lc}
           \eps^{2-\theta}  \left[\, \C_{2}\displaystyle\int _{\R^{n}}  \frac{U^{2}}{|x|^{\theta}} dx  + o(1)\right] \qquad    &\hbox{ if }~ \beta_{+}-\beta_{-} >2-\theta, \\
           $~$\\
           \eps^{2-\theta} \log \left( \dfrac{1}{\eps}\right) \left[\,\C_{2} \omega_{n-1} + o(1) \right] \qquad  &\hbox{ if  }~ \beta_{+}-\beta_{-} =2-\theta,\\~\\
          O\left(\eps^{\beta_{+}-\beta_{-}}\right) \qquad   &\hbox{ if }~ \beta_{+}-\beta_{-} <2-\theta.
            \end{array} \right.
\end{equation*}
And
\begin{equation*}
-\C_{1}\displaystyle\int _{\Omega} \frac{\log |x|}{|x|^{\theta}}u_{\eps}^2~dx  =
 \left \{ \begin{array} {lc}
          \C_{1} \eps^{2-\theta} \log \left( \frac{1}{\eps}\right) \left[ ~  \displaystyle\int _{\R^{n}}  \frac{U^{2}}{|x|^{\theta}} dx  + o(1)\right] \qquad    &\hbox{ if }~ \beta_{+}-\beta_{-} >2-\theta, \\~\\
           \C_{1}\eps^{2-\theta} \left( \log \left( \dfrac{1}{\eps}\right) \right)^{2} \left[ \,\dfrac{\omega_{n-1}}{2} + o(1) \right] \qquad  &\hbox{ if  }~ \beta_{+}-\beta_{-} =2-\theta,\\
           ~\\
          O\left(\eps^{\beta_{+}-\beta_{-}}\right) \qquad   &\hbox{ if }~ \beta_{+}-\beta_{-} <2-\theta.
            \end{array} \right.
\end{equation*}
\noindent
Note that $\displaystyle \beta_{+}-\beta_{-} \geq 2 -\theta$ if and only if  $  \gamma \leq  \frac{(n-2)^2}{4}-\frac{(2-\theta)^{2}}{4}$. Therefore,
\begin{align*}
\displaystyle\int _{\Omega} h(x)u_{\eps}^2~dx   =   \left \{ \begin{array} {lc}
                  \eps^{2-\theta} \displaystyle\int _{\R^{n}}  \frac{U^{2}}{|x|^{\theta}} dx \left[  ~\C_{1} \log \left( \frac{1}{\eps}\right) (1+o(1)) +\C_{2}+ o(1)\right]    $~$    &\hbox{ if }~ \gamma< \frac{(n-2)^2}{4}-\frac{(2-\theta)^{2}}{4}, \\
           $~$\\\displaystyle
           \eps^{2-\theta} \log \left( \frac{1}{\eps}\right) \frac{\omega_{n-1}}{2} \left[ ~\C_{1} \log \left( \frac{1}{\eps}\right) (1+o(1)) +2\C_{2}+ o(1) \right] $~$  &\hbox{ if  }~ \gamma= \frac{(n-2)^2}{4}-\frac{(2-\theta)^{2}}{4}.\\
            \end{array} \right.
\end{align*}
\noindent
Combining the above estimates, we obtain as $\eps\to 0$,
\[\begin{split}
J^{\Omega}_{\gamma,h}(u_{\eps}
&)=
    \frac{\displaystyle\int _{\Omega} \left(\absgrad{u_{\eps}}^2 -\gamma \frac{u_{\eps}^{2}}{|x|^{2}}- h(x) u_{\eps}^{2}\right) dx}{\left(~\displaystyle\int _{\Omega}  b(x) \frac{|u_{\eps}|^{\crits}}{|x|^{s}} \,dx\right)^{{2}/{\crits}}} \\
&=
    \frac{ \mu_{\gamma, 0}(\R^{n})}{b(0)^{2/\crits}}-
        \begin{cases}
        \dfrac{\displaystyle\int _{\R^{n}}  \frac{U^{2}}{|x|^{\theta}}\,dx}
        {\left( b(0) \displaystyle\int _{\R^{n}} \frac{U^{\crits}}{|x|^s}~ dx\right)^{2/\crits}} \eps^{ 2 -\theta}\left[\,\C_{1} \log \left( \dfrac{1}{\eps}\right) (1+o(1)) +\C_{2}+ o(1)\right]\\
            \hfill\textif\gamma<\frac{(n-2)^2}{4}-\frac{(2-\theta)^{2}}{4},\\
            ~\\
        \dfrac{ \omega_{n-1}}{2 \left( b(0) \displaystyle\int _{\R^{n}} \frac{U^{\crits}}{|x|^s}~ dx\right)^{2/\crits}} \eps^{ 2 -\theta} \log \left(\dfrac{1}{\eps}\right)\left[\,\C_{1} \log \left( \dfrac{1}{\eps}\right) (1+o(1)) +2\C_{2}+ o(1)\right]\\
            \hfill\textif\gamma=\frac{(n-2)^2}{4}-\frac{(2-\theta)^{2}}{4},
        \end{cases}
\end{split}\]
as long as $\beta_{+}-\beta_{-} \geq 2-\theta.$ Thus,  for $\eps$ sufficiently small, the assumption that either   $ \C_{1}>0$ or $\C_{1}=0$, $\C_{2}>0$ guarantees that
$$\mu_{\gamma,  h}(\Omega)\leq J_{\gamma,h}^\Omega(u_{\eps}) <\frac{ \mu_{\gamma,0}(\R^{n})}{b(0)^{2/\crits}}.$$
It then follows  from Lemma \ref{low energy compactness} that $ \mu_{\gamma,h}(\Omega)$ is attained.
\bigskip

\noindent
{\bf {Case 2:}} Test-functions for the case when $\displaystyle   \frac{(n-2)^2 }{4} - \frac{(2-\theta)^2}{4} < \gamma < \frac{(n-2)^2}{4}$.
\medskip

\noindent
Here $h(x)$  and $\theta$ given by \eqref{eq:singh} satisfy the hypothesis of Proposition \eqref{Propoaition -Regularity}.
 Since $\gamma> \frac{(n-2)^2}{4}-\frac{(2-\theta)^2}{4}$, it follows from  \eqref{def:mass} that there exists $\beta\in D^{1,2}(\Omega)$ such that
\begin{equation}\label{asymp:beta}
\beta(x)\simeq_{x\to 0}\frac{m_{\gamma, h}(\Omega)}{|x|^{\beta_{-}}}.
\end{equation}
The function $K(x):=\frac{\eta(x)}{|x|^{\beta_{+}}}+\beta(x)$ for $x\in \Omega\setminus\{0\}$ satisfies the equation:
\begin{equation}\label{eqnH}
\left\{\begin{array}{@{}ll}
-\Delta K-\left(\frac{\gamma}{|x|^2}+h(x)\right)K=0 &\hbox{ in }\Omega\setminus\{0\}\\
\hfill K>0 &\hbox{ in }\Omega\setminus\{0\}\\
\hfill K=0 &\hbox{ on }\partial\Omega.
\end{array}\right.
\end{equation}
\noindent
Define the test functions
$$u_{\eps}(x):= \eta(x) U_{\eps} + \eps^{\frac{\beta_{+}-\beta_{-}}{2}} \beta(x) \qquad  \hbox{ for } x \in \overline{\Omega}\setminus \{ 0\}$$
The functions $u_{\eps} \in  D^{1,2}(\Omega)$ for all $\eps >0$. We  estimate  $J_{\gamma, h}^\Omega(u_{\eps})$.
\bigskip

\noindent
{\bf Step 1:} Estimates for  $\displaystyle \displaystyle\int _{\Omega} \left( |\nabla u_{\eps}|^{2} - \left( \frac{\gamma}{|x|^{2}} +h(x)\right) u_{\eps}^{2}\right) dx$.
\medskip

\noindent
Take $\delta>0$ small enough such that $\eta(x)=1$ in $ B_\delta(0)\subset \Omega$.
We decompose the integral as
\begin{multline*}
\displaystyle\int _{\Omega} \left( |\nabla u_{\eps}|^{2} - \left( \frac{\gamma}{|x|^{2}} + h(x)\right) u_{\eps}^{2}\right) dx \\
=
\displaystyle\int _{B_{\delta}(0)} \left( |\nabla u_{\eps}|^{2} - \left( \frac{\gamma}{|x|^{2}} + h(x)\right) u_{\eps}^{2}\right) dx+\displaystyle\int _{\Omega \setminus B_{\delta}(0)} \left( |\nabla u_{\eps}|^{2} - \left( \frac{\gamma}{|x|^{2}} + h(x)\right) u_{\eps}^{2}\right) dx.
\end{multline*}
By standard elliptic estimates, it follows that
$\lim _{\eps \to 0}  \frac{u_{\eps}}{\eps^{\frac{\beta_{+}-\beta_{-}}{2} }}= K$ in $C^{2}_{\loc}(\overline{\Omega}\setminus \{ 0\}).$
Hence
\[\begin{split}
&\hspace{-1cm}\lim _{\eps \to 0}\frac{ \displaystyle\int _{\Omega\setminus B_{\delta}(0)} \left( |\nabla u_{\eps}|^{2} - \left( \frac{\gamma}{|x|^{2}} + h(x)\right) u_{\eps}^{2}\right) dx}{\eps^{\beta_{+}-\beta_{-}}}
=
\displaystyle\int _{\Omega\setminus B_{\delta}(0)}  \left( |\nabla K|^{2} - \left( \frac{\gamma}{|x|^{2}} + h(x)\right) K^{2}\right) dx \notag \\
&=
\displaystyle\int _{\Omega\setminus B_{\delta}(0)}  \left( - \Delta K -  \left( \frac{\gamma}{|x|^{2}} + h(x)\right)K\right) K~dx + \displaystyle\int _{\partial \left(\Omega\setminus B_{\delta}(0)\right)} K \partial_{\nu} K~ d\sigma \notag\\
&=
\displaystyle\int _{\partial \left(\Omega\setminus B_{\delta}(0)\right)} K \partial_{\nu} K~ d\sigma = -\displaystyle\int _{\partial  B_{\delta}(0)} K \partial_{\nu} K ~d\sigma.
\end{split}\]
Since $\beta_{+}+\beta_{-}=n-2$, using elliptic estimates, and the definition of $K$ gives us
$$K\partial_{\nu} K=-\frac{\beta_{+}}{|x|^{1+2\beta_{+}}}-(n-2) \frac{m_{\gamma, h}(\Omega)}{|x|^{n-1}}+o \left( \frac{1}{|x|^{n-1}} \right) \quad \hbox{as $x\to 0$.} $$
Therefore,
\begin{align*}
\displaystyle\int _{\Omega \setminus B_{\delta}(0)} \left( |\nabla u_{\eps}|^{2} - \left( \frac{\gamma}{|x|^{2}} +h(x)\right) u_{\eps}^{2}\right) dx= \eps^{\beta_{+}-\beta_{-}}\omega_{n-1}\left(\frac{\beta_{+}}{\delta^{\beta_{+}-\beta_{-}}}+(n-2)m_{\gamma, h}(\Omega)+o_\delta(1)\right)
\end{align*}
\medskip

\noindent
Now, we estimate the term   $\displaystyle \displaystyle\int _{B_{\delta}(0)} \left( |\nabla u_{\eps}|^{2} - \left( \frac{\gamma}{|x|^{2}} + h(x)\right) u_{\eps}^{2}\right) dx$. \\
First,  $u_{\eps}(x)=U_{\eps}(x)+\eps^{\frac{\beta_{+}-\beta_{-}}{2}}\beta(x)$ for $x\in B_\delta(0)$, therefore after integration by parts, we obtain
\[\begin{split}
 \displaystyle\int _{B_{\delta}(0)} \left( |\nabla u_{\eps}|^{2} - \left( \frac{\gamma}{|x|^{2}} +h(x)\right) u_{\eps}^{2}\right) dx
 &= \displaystyle\int _{B_{\delta}(0)} \left( |\nabla U_{\eps}|^{2}
 - \left( \frac{\gamma}{|x|^{2}} + h(x)\right) U_{\eps}^{2}\right) dx\\
 &\quad+2 \eps^{\frac{\beta_{+}-\beta_{-}}{2}}  \displaystyle\int _{B_{\delta}(0)} \left( \nabla U_{\eps}\cdot \nabla \beta - \left( \frac{\gamma}{|x|^{2}}  + h(x)\right) U_{\eps} \beta\right) dx \notag\\
 &\quad+ \eps^{\beta_{+}-\beta_{-}} \displaystyle\int _{B_{\delta}(0)} \left( |\nabla \beta|^{2} - \left( \frac{\gamma}{|x|^{2}} + h(x)\right) \beta^{2}\right) dx  \notag\\
 &= \displaystyle\int _{B_{\delta}(0)} \left(-\Delta  U_{\eps} -  \frac{\gamma}{|x|^{2}}  U_{\eps}\right) U_{\eps}~ dx  + \displaystyle\int _{\partial B_{\delta}(0)}  U_{\eps} \partial_{\nu} U_{\eps}~ d\sigma \\
 &\quad - \displaystyle\int _{B_{\delta}(0)} h(x) U_{\eps}^{2}~ dx +2  \eps^{\frac{\beta_{+}-\beta_{-}}{2}}  \displaystyle\int _{B_{\delta}(0)} \left(-\Delta  U_{\eps} ~ dx  -  \frac{\gamma}{|x|^{2}} U_{\eps}\right) \beta~ dx\\
 &\quad  -  2  \eps^{\frac{\beta_{+}-\beta_{-}}{2}}  \displaystyle\int _{B_{\delta}(0)}  h(x) U_{\eps}\beta  ~ dx\notag + 2 \eps^{\frac{\beta_{+}-\beta_{-}}{2}}  \displaystyle\int _{\partial B_{\delta}(0)} \beta \partial_{\nu}U_{\eps} ~ d\sigma \\
 &\quad+ \eps^{\beta_{+}-\beta_{-}} \displaystyle\int _{B_{\delta}(0)} \left( |\nabla \beta|^{2} - \left( \frac{\gamma}{|x|^{2}} +h(x)\right) \beta^{2}\right) dx.
\end{split}\]
\noindent
We now estimate each of the above terms. First, using  equation \eqref{eq:U1} and the expression for $U_{\eps},$ we obtain
\begin{align*}
 \displaystyle\int _{B_{\delta}(0)} \left(-\Delta  U_{\eps} -  \frac{\gamma}{|x|^{2}}  U_{\eps}\right) U_{\eps}~ dx &=\chi  \displaystyle\int _{B_\delta(0)} \frac{U_{\eps}^{\crits}}{|x|^s}~ dx \\
&=\chi \displaystyle\int _{\R^{n}} \frac{U^{\crits}}{|x|^s}~ dx+ O \left(\eps^{ \frac{\crits}{2}  (\beta_{+}-\beta_{-})} \right),
\end{align*}
and
\begin{align*}
\displaystyle\int _{\partial B_{\delta}(0)}  U_{\eps} \partial_{\nu} U_{\eps}~ d\sigma = -\beta_{+}\omega_{n-1} \frac{\eps^{\beta_{+}-\beta_{-}}}{\delta^{\beta_{+}-\beta_{-}}}
    +o_\delta\left(\eps^{\beta_{+}-\beta_{-}}\right) \quad \hbox{as $\eps \to 0$.}
\end{align*}
\medskip

\noindent
Note that $$\displaystyle \beta_{+}-\beta_{-} < 2-\theta \iff  \gamma >  \frac{(n-2)^2}{4}- \frac{(2-\theta)^2}{4} \implies 2 \beta_{+} +\theta< n.$$ Therefore,
\begin{align*}
& \displaystyle\int _{B_{\delta}(0)}  h(x)  U_{\eps}^{2}~ dx=O \left( \eps^{\beta_{+}-\beta_{-}}   \displaystyle\int _{B_{\delta}(0)} \frac{1}{|x|^{2\beta_{+}+\theta}} dx  \right)=o_\delta\left(\eps^{\beta_{+}-\beta_{-}}\right) \quad \hbox{as $\eps\to 0$.} \\
\end{align*}
\noindent
Again from equation \eqref{eq:U1} and the expression for $U$ and $\beta,$ we get that
\begin{align*}
\displaystyle\int _{B_{\delta}(0)} \left(-\Delta  U_{\eps} ~ dx -  \frac{\gamma}{|x|^{2}} U_{\eps}\right) \beta~ dx & =\eps^{\frac{\beta_{+}+\beta_{-}}{2}} \displaystyle\int _{B_{\eps^{-1}\delta}(0)} \left(-\Delta  U ~ dx -  \frac{\gamma}{|x|^{2}} U \right) \beta(\eps x)~ dx
\end{align*}
\begin{align*}
&= m_{\gamma, h}(\Omega) \eps^{\frac{\beta_{+}-\beta_{-}}{2}} \displaystyle\int _{B_{\eps^{-1}\delta}(0)} \left(-\Delta  U ~ dx -  \frac{\gamma}{|x|^{2}} U \right) |x|^{-\beta_{-}}~ dx + o_{\delta} \left(  \eps^{\frac{\beta_{+}-\beta_{-}}{2}} \right) \\
&= m_{\gamma, h}(\Omega) \eps^{\frac{\beta_{+}-\beta_{-}}{2}} \displaystyle\int _{B_{\eps^{-1}\delta}(0)} \left(-\Delta  |x|^{-\beta_{-}} ~ dx -  \frac{\gamma}{|x|^{2}} |x|^{-\beta_{-}} \right) U~ dx\\
& \quad -  m_{\gamma, h}(\Omega) \eps^{\frac{\beta_{+}-\beta_{-}}{2}}\displaystyle\int _{\partial B_{\eps^{-1} \delta}(0)}   \frac{\partial_{\nu} U}{|x|^{\beta_{-}}}~ d\sigma + o_{\delta} \left(  \eps^{\frac{\beta_{+}-\beta_{-}}{2}} \right) \\
&= \beta_{+} m_{\gamma, h}(\Omega) \omega_{n-1}\eps^{\frac{\beta_{+}-\beta_{-}}{2}}+o_{\delta} \left(  \eps^{\frac{\beta_{+}-\beta_{-}}{2}} \right) .
\end{align*}
Similarly,
\begin{align*}
 \displaystyle\int _{\partial B_{\delta}(0)} \beta \partial_{\nu}U_{\eps} ~ d\sigma =- \beta_{+} m_{\gamma, h}(\Omega)\omega_{n-1}\eps^{\frac{\beta_{+}-\beta_{-}}{2}}+o_{\delta} \left(  \eps^{\frac{\beta_{+}-\beta_{-}}{2}} \right).
\end{align*}
Since  $\beta_{+}+\beta_{-}+\theta=n-(2-\theta) <n,$ we have
\begin{align*}
  \displaystyle\int _{B_{\delta}(0)}  h(x) U_{\eps}\beta ~ dx &= O \left( \eps^{\frac{\beta_{+}-\beta_{-}}{2}}  \displaystyle\int _{B_{\delta}(0)}  \frac{1}{|x|^{\beta_{+}+\beta_{-} +\theta}} ~dx \right) \\
  &= o_{\delta}\left( \eps^{\frac{\beta_{+}-\beta_{-}}{2}}  \right).
\end{align*}
\medskip

\noindent
And, finally
\begin{align*}
\eps^{\beta_{+}-\beta_{-}} \displaystyle\int _{B_{\delta}(0)} \left( |\nabla \beta|^{2} - \left( \frac{\gamma}{|x|^{2}}  + h(x)\right) \beta^{2}\right) dx=o_\delta(\eps^{\beta_{+}-\beta_{-}}).
\end{align*}
\medskip
\noindent
Combining all the estimates, we get
\begin{align*}
\displaystyle\int _{B_{\delta}(0)} \left( |\nabla u_{\eps}|^{2} - \left( \frac{\gamma}{|x|^{2}} +h(x)\right) u_{\eps}^{2}\right) dx   =\chi \displaystyle\int _{\R^{n}} \frac{U^{\crits}}{|x|^s}~ dx-\beta_{+}\omega_{n-1} \frac{\eps^{\beta_{+}-\beta_{-}}}{\delta^{\beta_{+}-\beta_{-}}}+o_\delta(\eps^{\beta_{+}-\beta_{-}}).
\end{align*}
\medskip

\noindent
So,
\begin{align}
\displaystyle\int _{\Omega} \left( |\nabla u_{\eps}|^{2} - \left( \frac{\gamma}{|x|^{2}} +h(x)\right) u_{\eps}^{2}\right) dx =& \chi \displaystyle\int _{\R^{n}} \frac{U^{\crits}}{|x|^s}~ dx +  \omega_{n-1} (n-2)m_{\gamma, h}(\Omega) \eps^{\beta_{+}-\beta_{-}} + o_\delta(\eps^{\beta_{+}-\beta_{-}}).\notag
\end{align}
\medskip

\noindent
{\bf Step 2:} Estimating  $\displaystyle \displaystyle\int _\Omega b(x) \frac{u_{\eps}^{\crits}}{|x|^s}~ dx.$
\medskip

\noindent
One has  for $\delta >0$ small
\begin{align}
\displaystyle\int _\Omega b(x) \frac{u_{\eps}^{\crits}}{|x|^s}~ dx &= \displaystyle\int _{B_\delta(0)} b(x)\frac{u_{\eps}^{\crits}}{|x|^s}~ dx+ \displaystyle\int _{ \Omega \setminus B_\delta(0)} b(x) \frac{u_{\eps}^{\crits}}{|x|^s}~ dx \notag  \\
&= \displaystyle\int _{B_\delta(0)} b(x)\frac{  \left( U_{\eps}(x)+\eps^{\frac{\beta_{+}-\beta_{-}}{2}} \beta(x) \right)^{\crits}}{|x|^s}~ dx+ o(\eps^{\beta_{+}-\beta_{-}}) \notag \\
&= \displaystyle\int _{B_\delta(0)} b(x)\frac{U_{\eps}^{\crits}}{|x|^s}~ dx+\eps^{\frac{\beta_{+}-\beta_{-}}{2}} \crits  \displaystyle\int _{B_\delta(0)} b(x)\frac{U_{\eps}^{\crits-1}}{|x|^s}\beta ~ dx\notag  \\
&\quad + o(\eps^{\beta_{+}-\beta_{-}}) \notag \\
&= \displaystyle\int _{B_\delta(0)} b(x)\frac{U_{\eps}^{\crits}}{|x|^s}~ dx+ \eps^{\frac{\beta_{+}-\beta_{-}}{2}}\frac{\crits}{\chi}  \displaystyle\int _{B_{\delta}(0)} b(x) \left(-\Delta  U_{\eps} ~ dx -  \frac{\gamma}{|x|^{2}} U_{\eps}\right) \beta~ dx \notag \\
&\quad +o(\eps^{\beta_{+}-\beta_{-}})\notag  \\
&=b(0) \displaystyle\int _{\R^{n}} \frac{U^{\crits}}{|x|^s}~ dx + \frac{\crits}{\chi} b(0) \beta_{+} m_{\gamma, \lambda, a}(\Omega) \omega_{n-1}\eps^{\beta_{+}-\beta_{-}}+o(\eps^{\beta_{+}-\beta_{-}}).
\end{align}
\medskip

\noindent
So, we obtain
\begin{align}
J^{\Omega}_{\gamma,\lambda,a}(u_{\eps})&=  \frac{\displaystyle\int _{\Omega} \left(\absgrad{u_{\eps}}^2 -\gamma \frac{u_{\eps}^{2}}{|x|^{2}}-h(x) u_{\eps}^{2}\right) dx}{\left(~\displaystyle\int _{\Omega}  b(x) \frac{|u_{\eps}|^{\crits}}{|x|^{s}} \,dx\right)^{{2}/{\crits}}} \notag \\
&=\frac{ \mu_{\gamma, 0}(\R^{n})}{b(0)^{2/\crits}}- m_{\gamma, h}(\Omega)\frac{\omega_{n-1} (\beta_{+}-\beta_{-}) }{\left( b(0) \displaystyle\int _{\R^{n}} \frac{U^{\crits}}{|x|^s}~ dx\right)^{2/\crits}}  \eps^{\beta_{+}-\beta_{-}} + o(\eps^{\beta_{+}-\beta_{-}}).
\end{align}
Therefore,   if $m_{\gamma, h}(\Omega)>0$, we get for $\eps$ sufficiently small
$$\mu_{\gamma,  h}(\Omega)\leq J_{\gamma, h}^\Omega(u_{\eps}) <\frac{ \mu_{\gamma,0}(\R^{n})}{b(0)^{2/\crits}}.$$
Then, from Lemma \ref{low energy compactness} it follows that $ \mu_{\gamma,h}(\Omega)$ is attained.
$\hfill \Box$

\medskip

\begin{remark}  Assume for simplicity that $h(x)= \lambda |x|^{-\theta}$ where $0 \leq \theta<2$. There is a threshold  $\lambda^{*}(\Omega) \geq 0$ beyond which the infimum $ \mu_{\gamma,\lambda}(\Omega)$ is achieved, and below which, it is not. In fact,
$$\lambda^{*}(\Omega):= \sup\{\lambda:  \mu_{\gamma,\lambda }(\Omega) =  \mu_{\gamma,0 }(\R^{n})\}.$$
Performing a blow-up analysis  like in  \cite{Ghoussoub--Robert-remaining cases} one can obtain the following sharp results:

\begin{itemize}
\item
In high dimensions, that is for  $ \gamma \leq \frac{(n-2)^{2}}{4}-\frac{(2-\theta)^{2}}{4} $ one has  $\lambda^{*}(\Omega)=0$ and the  infimum $ \mu_{\gamma,\lambda }(\Omega)$ is achieved if and only if $\lambda > \lambda^{*}(\Omega)$.\\
\item
In low  dimensions, that is for  $\frac{(n-2)^{2}}{4}-\frac{(2-\theta)^{2}}{4} <\gamma $,  one has  $\lambda^{*}(\Omega)>0$   and  $ \mu_{\gamma,\lambda}(\Omega)$ is not achieved  for  $\lambda < \lambda^{*}(\Omega)$ and $ \mu_{\gamma,\lambda}(\Omega)$ is achieved for $\lambda > \lambda^{*}(\Omega)$. Moreover
under the assumption  $ \mu_{\gamma,\lambda^{*}}(\Omega)$ is not achieved,  we have that $m_{\gamma, \lambda^{*} }(\Omega)=0$, and $ \lambda^{*}(\Omega)= \sup \{\lambda :  m_{\gamma, \lambda }(\Omega) \leq 0\}$.
\end{itemize}

\end{remark}

\section{Existence results for compact submanifolds  of $\Bn$ }

Back to the following Dirichlet boundary value problem in hyperbolic space. Let $\Omega \Subset \Bn$ ($n \geq 3$) be a bounded smooth domain such that $0 \in \Omega$. 
We consider the Dirichlet boundary value problem:
\begin{equation} \label{hyperHSeqn}
\left\{ \begin{array}{@{}l@{\;}ll}
-\Delta_{\Bn}u-\gamma{V_2}u -\lambda u&=V_{\crits}u^{\crits-1} &\textin\Omega \\
\hfill u &\geq 0  &\textin\Omega  \\
\hfill u &=0   & \hbox{ on }  \partial \Omega,
\end{array} \right.
\end{equation}
where $\lambda \in \R$, $0<s<2$ and $\gamma <\gamma_H:= \frac{(n-2)^2}{4}.$

We shall use the conformal transformation $ g_{\Bn}= \varphi^{\frac{4}{n-2}}g_{\hbox{Eucl}}$, where $ \varphi= \left( \frac{2}{1-r^{2}} \right)^{\frac{n-2}{2}} $ to map the problem into $\R^n$.
We start by  considering the general equation :
\begin{align}\label{Eqt:hyperbolic H-S operator = f(x,u)}
-\Delta_{\Bn}u-\gamma{V_2}u-\lambda u=f(x,u)  \quad\textin\Omega \Subset \Bn,
\end{align}
where $f(x,u)$ is a Carath\'{e}odory function such that
$$ | f(x,u) | \le C|u| \left(  1 + \frac{|u|^{\crits-2}}{r^s}\right)  \quad  \text{ for all
} x \in \Omega.$$
If $u$ satisfies (\ref{Eqt:hyperbolic H-S operator = f(x,u)}), then
 ${ v:= \varphi u}$ satisfies the equation:
\begin{align*}
-\Delta v  -\gamma \left(\frac{2}{1-r^{2}}\right)^{2} V_{2} v-  \left[\lambda -\frac{n(n-2)}{4}  \right]  \left(\frac{2}{1-r^{2}}\right)^{2}  v = \varphi^{\frac{n+2}{n-2}} f\left(x, \frac{v}{\varphi} \right)  \quad\textin\Omega.
\end{align*}
On the other hand, we have the following expansion for $\left(\frac{2}{1-r^{2}}\right)^{2} V_2:$
\begin{align*}
\left(\frac{2}{1-r^{2}}\right)^{2} V_{2}(x)=  \frac{1}{(n-2)^{2}}\left( \frac{f(r)^{}}{ G(r)} \right)^{2}
\end{align*}
where $f(r)$ and $G(r)$ are given by \eqref{hypgreen}. We then obtain  that
\begin{align}
\left(\frac{2}{1-r^{2}}\right)^{2} V_{2}(x)=
\left \{ \begin{array} {lc}
           \frac{1}{r^{2}} + \frac{4}{r} \left( \frac{1}{1-r}\right)+ \frac{4}{(1-r)^{2}} \qquad    &\hbox{ when }~ n=3, \\
           $~$\\
               \frac{1}{r^{2}} +8 \log \frac{1}{r}-4 +g_{4}(r) \qquad  &\hbox{ when }~ n=4, \\
           $~$\\
              \frac{1}{r^{2}} +\frac{4(n-2)}{n-4}+ r g_{n}(r)  \qquad   &\hbox{ when }~ n\geq5.
            \end{array} \right.
\end{align}
where  for all $n \geq 4$, $g_{n}(0)=0$ and $g_{n}$ is $C^{0}([0,\delta])$ for $\delta <1$.
\medskip

\noindent
This implies that $v:= \varphi u$ is a solution to
\begin{align*}
-\Delta v  - \frac{\gamma}{r^{2}} v- \left[ \gamma a(x) + \left(\lambda - \frac{n(n-2)}{4}  \right) \left(\frac{2}{1-r^{2}}\right)^{2} \right] v= \varphi^{\frac{n+2}{n-2}} f\left(x, \frac{v}{\varphi} \right).
\end{align*}
where $a(x)$ is defined in \eqref{def:a}.
We can therefore state the following lemma:

\begin{lemma}\label{Lemma:relation between Hyp and Euc solution on bounded domain}
A non-negative function $u \in H_0^1(\Omega)$ solves  \eqref{hyperHSeqn} if and only if  ${ v:= \varphi u} \in H^1_0(\Omega)$ satisfies
\begin{equation} \label{Problem on bounded domain:Euclidean}
\left\{ \begin{array}{@{}l@{\;}ll}
-\Delta v  - \left( \frac{\gamma}{|x|^{2}} + h_{\gamma,\lambda}(x)\right)  v &=b(x) \frac{v^{\crits-1}}{|x|^{s}} &\textin\Omega \\
\hfill v &\geq0 & \textin \Omega\\
\hfill v &=0   & \hbox{ on }  \partial \Omega,
\end{array} \right.
\end{equation}
where    
$$ \displaystyle h_{\gamma,\lambda}(x)= \gamma a(x)+\frac{4\lambda - {n(n-2)}}{(1-|x|^{2})^{2}},
$$
and
$a(x)$ is defined  in \eqref{def:a},   
and $b(x)$ is a positive function in $C^1(\overline{\Omega})$  with $  \displaystyle b(0)=\frac{(n-2)^{\frac{n-s}{n-2}}}{2^{2-s}} ~ \hbox{ and }~ \nabla b(0)=0$.
Moreover, the  hyperbolic operator  $\LB : = \displaystyle -\Delta_{\Bn}-\gamma{V_2}-\lambda$ is coercive if and only if the corresponding Euclidean operator $\displaystyle L^{\R^n}_{\gamma, h}:=  -\Delta   - \left( \frac{\gamma}{|x|^{2}} + h_{\gamma,\lambda}(x)  \right)   $ is coercive.

\end{lemma}

\begin{proof}
Note that  one has  in particular
\begin{align}\label{def:h}
 h_{\gamma,\lambda}(x) =h_{\gamma,\lambda}(r)=
\left \{ \begin{array} {lc}
        \frac{4 \gamma}{r} +4 \gamma  + \frac{4\gamma-3}{(1-r)^{2}} +  \frac{4 \gamma }{1-r} r \quad    &\hbox{ when }~ n=3, \\
           $~$\\
              \left[8\gamma \log \frac{1}{r}-4\gamma +4 \lambda -8 \right]  \\
                \quad+\gamma g_{4}(r) +  (4 \lambda-8)\frac{ r^{2}(2-r^{2})}{(1-r^{2})^{2}}\quad  &\hbox{ when }~ n=4,\\
           $~$\\
               \frac{4(n-2) }{n-4}   \left[ \frac{n-4}{n-2}\lambda +\gamma-\frac{n(n-4)}{4} \right]\\
                \quad+\gamma r g_{n}(r) + (4\lambda -n(n-2))\frac{ r^{2}(2-r^{2})}{(1-r^{2})^{2}}  \quad   &\hbox{ when }~ n\geq5,
            \end{array} \right.
\end{align}
with $g_{n}(0)=0$ and $g_{n}$ is $C^{0}([0,\delta])$ for $\delta <1$, for all $n \geq 4$.
\medskip

Let $f(x,u) = V_{\crits} u^{\crits-1}$ in (\ref{Eqt:hyperbolic H-S operator = f(x,u)}). The above remarks show that $v:= \varphi u$ is a solution to (\ref{Problem on bounded domain:Euclidean}).\\
For the second part, we first note that the following identities hold:
$$ \displaystyle\int _{\Omega} \left( |\nabla_{\Bn}u|^{2}-\frac{n(n-2)}{4}u^{2}\right)~dv_{g_{\Bn}}= \displaystyle\int _{\Omega}|\nabla v|^{2}~dx $$
and
$$ \displaystyle\int _{\Omega} u^{2}dv_{g_{\Bn}}= \displaystyle\int _{\Omega}   v^{2} \left(\frac{2}{1-r^{2}} \right)^{2}~dx.$$
If the operator $\LB$ is coercive, then for any $u \in C^\infty(\Omega),$ we have $\langle \LB u, u \rangle \ge  C \|u\|^{2}_{H_0^1(\Omega)},$ which means
$$ \displaystyle\int _{\Omega} \left( |\nabla_{\Bn}u|^{2}-\gamma{V_2}u^{2}\right)dv_{g_{\Bn}}  \geq C   \displaystyle\int _{\Omega} \left( |\nabla_{\Bn}u|^{2}+u^{2}\right)dv_{g_{\Bn}}.$$
 The latter then holds if and only if
 \begin{align*}
\langle \displaystyle L^{\R^n}_{\gamma, \phi} u , u \rangle &=  \displaystyle\int _{\Omega} \left( |\nabla v|^{2}-\left(\frac{2}{1-r^{2}} \right)^{2} \left(\gamma{V_2} - \frac{n(n-2)}{4} \right) v^{2}  \right)dx  \\
 &\geq C   \displaystyle\int _{\Omega} \left( |\nabla v|^{2}+\left(\frac{2}{1-r^{2}} \right)^{2} \left( \frac{n(n-2)}{4}+1\right)v^{2}\right)dx \\
 &\geq   C'   \displaystyle\int _{\Omega} \left( |\nabla v|^{2}+v^2 \right)dx \ge c  \|u\|^{2}_{H^1_0({\Omega})},
\end{align*}
where $v= \varphi u$ is in $C^{\infty}(\Omega)$.  
This completes the proof.
\end{proof}
\medskip
\noindent

One can then use the results obtained in the last section to prove  Theorems  \ref{regularity}, \ref{hypermassI} and \ref{Thm:Main Result} stated in the introduction for the hyperbolic space.
Indeed, it suffices to consider equation (\ref{Problem on bounded domain:Euclidean}), where $b $ is a positive function in $C^1(\overline{\Omega})$ satisfying (\ref{def:b}) and  $h_{\gamma,\lambda}$ is given by \eqref{def:h}.

\medskip

If $n\geq 5$, then $\lim_{|x|\to 0}h_{\gamma, \lambda}(x)= \frac{4(n-2) }{n-4}   \left[ \frac{n-4}{n-2}\lambda +\gamma-\frac{n(n-4)}{4} \right]$, which is positive provided
$$\displaystyle \lambda  > \frac{n-2}{n-4} \left(\frac{n(n-4)}{4}-\gamma \right).$$
Moreover, since in this case $\theta =0$, Theorem \ref{Thm:minexist:Euclidean} holds when
$$\gamma \leq \frac{(n-2)^{2}}{4}-1= \frac{n(n-4)}{4}.$$

\medskip

If $n=3$, then $\lim\limits_{|x|\to 0}|x|h_{\gamma, \lambda}(x)=4\gamma$, hence $\gamma$ needs to be positive. On the other hand, since we use $\theta=1$, the first option in Theorem \ref{Thm:minexist:Euclidean} cannot occur and to have positive solutions one needs that the mass  $m^H_{\gamma, \lambda}(\Omega)$ to be positive. We note that the mass $m^H_{\gamma, \lambda}(\Omega)$ associated to the operator $\LB$ is a positive multiple of mass of the corresponding Euclidean operator. In other words, they both have the same sign.

\medskip

Similarly, for $n=4$, we have that $\lim\limits_{|x|\to 0}\dfrac{h_{\gamma, \lambda}(x)}{\log \frac{1}{|x|}}=8\gamma$. Hence $\gamma$ needs to be positive. On the other hand, $\gamma$ needs to be less than $\frac{(4-2)^{2}}{4}-1=0$, which is not possible. Hence  the first option in Theorem \ref{Thm:minexist:Euclidean} cannot occur, and again one needs that the mass  $m^H_{\gamma, \lambda}(\Omega)$ be positive.

\end{document}